\documentclass[a4paper,11pt,leqno]{article}
\makeatletter
\usepackage{amsfonts,delarray,amssymb,amsmath,amsthm,graphicx,color,epsfig,a4,a4wide}
\usepackage{amsmath,euscript}
\usepackage[ansinew]{inputenc}

\newtheorem{theo}{Theorem}
\newtheorem{pro}{Proposition}[section]
\newtheorem{lem}[pro]{Lemma}
\newtheorem{coro}[pro]{Corollary}
\newtheorem{remark}[pro]{Remark}
\newtheorem{defi}[pro]{Definition}

\def\Xint#1{\mathchoice
   {\XXint\displaystyle\textstyle{#1}}%
   {\XXint\textstyle\scriptstyle{#1}}%
   {\XXint\scriptstyle\scriptscriptstyle{#1}}%
   {\XXint\scriptscriptstyle\scriptscriptstyle{#1}}%
   \!\int}
\def\XXint#1#2#3{{\setbox0=\hbox{$#1{#2#3}{\int}$}
     \vcenter{\hbox{$#2#3$}}\kern-.5\wd0}}
\def\dashint{\Xint-}

\DeclareMathOperator{\supp}{Supp}

\def\({\left(}
\def\){\right)}
\def\1{\mathbf{1}}

\def\a{\alpha}

\def\curl{{\rm curl\,}}

\def\dr{{\delta_\mr}}

\def\div{\mathrm{div} \ }
\def\dt0{{{\frac{d}{dt}}_{|t=0}}}
\def\D{\displaystyle}

\def\E{\mathbb{E}}
\def\ep{\varepsilon}

\def\hal{\frac{1}{2}}

\def\indic{\mathbf{1}}

\def\loc{{\text{\rm loc}}}

\def\l|{\left|}

\def\mc{\mathbb{C}}

\def\mr{\mathbb{R}}
\def\mz{\mathbb{Z}}
\def\mo{{\mu_0}}
\def\nab{\nabla}
\def\np{\nab^{\perp}}

\def\p{\partial}

\def\ro{\rho}
\def\r|{\right|}

\def\sm{\setminus}
\def\supp{\text{Supp}}

\def\T{\mathbb{T}}
\def\X{\mathcal{X}}

\def\Im{Im}
\def\Re{Re}
\def\vp{\varphi}

\def\V{\mathrm{Var}}

\def\z{\zeta}
\def\Wc{\mathcal{W}_N}

 \numberwithin{equation}{section}
\title{Renormalized energy concentration in random matrices }
\author{Alexei Borodin and Sylvia Serfaty}

\begin{document}

\maketitle

\begin{abstract}
We define a ``renormalized energy" as an explicit functional on
arbitrary point configurations of constant average density in the plane and on the real line. The definition is inspired by ideas of 
\cite{ss1,ma1d}.
Roughly speaking, it is obtained by subtracting two leading terms from the Coulomb potential on a 
growing number of charges. The functional is expected to be a good measure of disorder of a 
configuration of points. We give certain formulas for its expectation for general stationary 
random point processes. For the random matrix $\beta$-sine processes on the real line ($\beta=1,2,4$), 
and Ginibre point process and zeros of Gaussian analytic functions process in the plane, 
we compute the expectation explicitly. Moreover, we prove that for these processes the variance of the 
renormalized energy vanishes, which shows concentration near the expected value.
We also prove that the $\beta=2$ sine process minimizes the renormalized energy in the
class of determinantal point processes with translation invariant correlation kernels.
\end{abstract}

\noindent
{\bf keywords:} renormalized energy, random point processes, random matrices, sine processes, Ginibre ensemble, zeros of Gaussian analytic functions. 

\tableofcontents

\section{Introduction}

The aim of this paper is to  introduce and compute a function, called the  ``renormalized energy",  for some specific random point processes that arise in random matrix models, and in this way to associate to each of these processes a unique number, which is expected to measure its ``disorder".

Our ``renormalized energy", that we denote $\mathcal{W}$, is defined over configurations of points lying either on the real line or on the plane,  as the limit  as $N \to \infty$ of
\begin{align*}
\Wc(\{a_i\})  &  = - \frac{1}{N} \sum_{i\neq j, a_i, a_j \in  [0, N]} \log \left|2\sin \frac{\pi(a_i-a_j)} N\right|  + \log N\qquad \text{in dimension 1},\\
\Wc(\{a_i\}) &  = \frac{1}{2\pi N^2 }  \sum_{i\neq j, a_i, a_j \in [0,N]^2} E_N(a_i-a_j) + \log \frac{N}{2\pi \eta(i)^2}\qquad \text{in dimension 2},\end{align*} where $E_N$ is an explicit Eisenstein series, and $\eta$ is the Dedekind Eta function.

This definition is inspired by that of the ``renormalized energy", denoted $W$, 
  introduced by Sandier and the second author  in \cite{ss1} in the case of points in the plane and in \cite{ma1d} in the case of points on the real line. The definitions for $W$ and $\mathcal{W}$ coincide when the point configuration has some periodicity (this is where our new definition originates), and in that case they amount to computing a sum of pairwise interactions
$$\sum_{i\neq j} G(a_i-a_j)$$ where $a_i$ are the points and $G$ is a suitable logarithmic kernel (the Green's function on the underlying torus);  however they are not identical in general. We will give more details on the connection in Section \ref{secdefW}.




In \cite{ma1d} it is  shown that in dimension 1,  $W$ is bounded below and its minimum is achieved at the perfect lattice $\mz$. In dimension 2, the situation is more complex; it is also shown in \cite{ss1} that the minimum of $W$ is achieved, but it is only conjectured that this minimum value is achieved at the perfect triangular lattice or ``Abrikosov lattice" according to the terminology of the physics of superconductors (which was the first main motivation for  $W$ in \cite{ss1} where it was introduced). This conjecture is supported by the result that, among configurations of points which form a perfect lattice (of fixed volume), the renormalized energy is minimal if and only if the lattice is the perfect triangular lattice, i.e. with 60$^\circ$ angles (that result is shown in \cite{ss1} based on the use of  modular functions and results in number theory).

It is thus natural to think of $W$ or $\mathcal{W}$ as a way to measure the disorder of a configuration of points. With this outlook, in dimension 1 the lattice $\mz$ is the most ordered configuration of points with prescribed density 1, while in dimension 2, it is expected to be the triangular lattice (which is better ordered than any other lattice, say the square lattice).
In addition, due to its logarithmic nature, $\mathcal{W}$ has some nice scaling and additive properties, which we believe make it a very good object.
A further motivation for choosing to study  $\mathcal{W}$ as opposed to any other total pairwise interaction function is 
that, as seen in \cite{ma2d} and \cite{ma1d}, $W$ arises very naturally from the statistical mechanics of Coulomb or log gases, which contain as particular cases the Ginibre and GOE/GUE/GSE ensembles of random matrices.
In \cite{ss1}, $W$ was introduced and derived in the context of the minimization of the  Ginzburg-Landau model of superconductivity. In \cite{ma2d} it was derived as a sort of limiting interaction energy for two dimensional Coulomb gases, and similarly in \cite{ma1d} with log gases.  These works are  based on analysis and energy estimates (upper and lower bounds). Both the questions we pursue here and the methods we use are quite different: they aim at obtaining explicit formulas for specific random matrix models. In particular, this is a way to compute some interesting statistics over random matrix eigenvalues, as initiated by Dyson-Mehta \cite{dysonmehta}. We will comment more on this just below.

\medskip

Let us now briefly introduce the notion of a  random point process.

A (simple) random point process is a probability measure on the set of all locally finite collections of (mutually distinct) points in the space, 
cf. e.g. \cite{dvj}.
It can also  be  viewed as a random  measure of the form $\xi(\omega)=\sum_{p \in \Lambda} \delta_p$, with the points $p$ distinct and $\Lambda $ discrete. 

Random point processes are essentially characterized by their ``$k$-point correlation functions" $\ro_k(x_1, \dots, x_k)$, which give the probability densities of finding $k$ points at the  locations $x_1, \dots, x_k$. 
We will normalize our processes so that the average number of points per unit volume is always $1$, which is equivalent to $\ro_1(x)\equiv 1$.

Perhaps the most famous random point process is the Poisson process, characterized by the facts that the number of points in disjoint subsets are independent, and the number of points in any finite volume subset of the space follows a Poisson distribution with parameter equal  to the volume of the set with respect
to a reference measure.
An important class of point processes is that of determinantal point processes, see \cite{Sos00}, \cite{Lyo03}, \cite{Joh05}, \cite{Kon05},
\cite{Hou06}, \cite{Sos06}, \cite{Bor11} and references therein. That class is characterized by the fact that  the $k$-point correlation functions are given by symmetric minors of a (correlation) kernel, cf. Section \ref{correl}.

It is easy to see that $\lim_{N\to\infty}\Wc=+\infty$ for the translation invariant Poisson process, which
means that it is ``too chaotic'' from the point of view of the renormalized energy. Let us list the (stationary) 
processes for which we show that $\mathcal W$ provides more meaningful information.

We will be interested in one dimension in the $\beta$-sine processes ($\beta=1,2,4$) which arise as the local limit of the law of eigenvalues in random matrix ensembles 
with orthogonal, unitary, and symplectic symmetry groups (they are determinantal for $\beta=2$ and Pfaffian otherwise). In two dimensions we will examine the ``Ginibre" point process, which is also a determinantal process arising as the local limit of the law of eigenvalues of matrices from the complex Ginibre ensemble (i.e. square matrices with complex Gaussian iid entries), for further reference, see \cite{forrester,agz,mehta}; as well as  the random process of zeros of random Gaussian analytic functions, often denoted GAF, whose description can be found in \cite{Hou06}.

As our processes are always translation invariant, the  $2$-point correlation function can always be written in the form $\ro_2(x,y)=1-T_2(x-y)$ for the
2-point \emph{cluster function} $T_2$ (we will come back to this notation in Sections \ref{correl} and   \ref{sec4}).

The main results we obtain are the following:
\begin{itemize}
\item For general stationary processes we identify sufficient conditions on the process and its $2$-point correlation function $\ro_2$  for the existence of $\lim_{N\to \infty} \E \Wc$, 
 and give an explicit  formula in terms of $\ro_2$ which is  (up to constants)
$$\lim_{N\to \infty}\E\Wc= \int_{\mr^d} \log |x|T_2(x)\, dx,$$ with $d=1 $ or $2$ according to the dimension.
\footnote{Interestingly enough, the sufficient conditions involve the equality $\int_{\mr^d} \log |x|T_2(x)\, dx=1$
that we check for the above mentioned processes. We expect
it to hold for general $\beta$-sine processes but we do not see an \emph{a priori} reason for that. 
It is also not clear to us whether this condition has a physical meaning.}

\item We apply this formula to the specific point processes mentioned above.
\item For the specific point processes above, we explicitly compute  the limit of the  variance of $\Wc$ and obtain that it is $0$. This implies that for such processes, $\Wc$ concentrates around its expected value, and converges in probability to $\lim\E\Wc$ as $N \to \infty$.
\item We prove that in the class of determinantal point processes with translation invariant kernels in dimensions
1 and 2, $\lim_{N\to \infty}\E\Wc$ is minimized by the processes whose correlation kernel is the Fourier transform of 
the characteristic function of the ball.  A complete physical interpretation of this fact seems to be missing. 
 In dimension 1 the optimization gives the $\beta=2$ sine process, which can be seen as a heuristic
explanation of the fact that this process is the universal local limit of $\beta=2$ random matrix ensembles.
Indeed, such a local limit has to be translation invariant and determinantal, and it is natural to expect that it also minimizes an appropriate energy functional.  In other dimensions,  point processes with such  specific kernels called ``Fermi-sphere point processes" appear in \cite{t} as higher-dimensional analogues of the sine process.

\end{itemize}
For our set  of specific processes, we thus show that we can attribute to the process a unique number, which we compute explicitly. Whenever these numbers are distinct this implies that  the processes are mutually singular (as measures
on the space of point configurations).
 Moreover, we check that the processes that are expected to have the highest level of ``order" or rigidity indeed have a lower value of $\lim_{N\to \infty} \E\Wc$.
For example for the $\beta$-sine processes we find 
$$\begin{array}{ll}
 \lim_{N\to \infty} \mathbb{E}\Wc= 2-\gamma - \log 2  &\quad  \text{for } \ \beta=1\\
 \lim_{N\to \infty} \mathbb{E}\Wc= 1-\gamma & \quad \text{for } \ \beta=2\\
\lim_{N\to \infty} \mathbb{E}\Wc=  \frac{3}{2}- \log 2 - \gamma & \quad \text{for } \ \beta=4,\end{array}$$
 where $\gamma$ is the Euler constant. These three numbers  form a decreasing sequence, as can be expected. 

 In two dimensions, we obtain that the Ginibre process has a higher $\lim_{N\to \infty}  \Wc$, hence less rigidity, than that of zeros of Gaussian 
analytic functions, in agreement with the results and philosophy of \cite{gnps}.
 
The values of $\lim \mathbb{E}\Wc$ for the $\beta=1,2,4$ sine processes above equal twice the thermodynamic
``energy per particle'' or ``mean internal energy'' for the log-gas with infinitely many particles, as obtained by Dyson in 1962 \cite[I.IX, III.VI]{dyson}. 
This is not surprizing as our definition of $\Wc$ essentially coincides with that of Dyson \cite[I.VI]{dyson} once the points on the interval $[0, N]$ 
are identified with points on a circle of same length. Furthermore, one can use
this fact to infer finer asymptotic properties of $\Wc$ as follows (we are very grateful to Peter 
Forrester for the idea).
If, instead of considering growing windows, one approximates the $\beta$-sine 
processes by the \emph{circular ensembles} of particles on the unit circle with joint probability density 
$\prod_{i<j}|z_i-z_j|^\beta$, one observes that in rescaled variables, $\Wc$ is simply $-2N^{-1}\sum_{i\ne j}\log|z_i-z_j|+\log N$,
and its characteristic function can be immediately obtained from Selberg's formula for the partition function:
$$
\mathbb{E} \exp(it\,\Wc^\mathrm{circular})=N^{it}\,\frac{Z(\beta-\frac{2it}N)}{Z(\beta)}\,, \qquad Z(\beta)=\frac{\Gamma(1+\frac{\beta N}2)}
{(\Gamma(1+\frac\beta 2))^N}\,,
$$
see e.g.
\cite[Section 4.7.2]{forrester} for the formula for $Z(\beta)$. 
Using Stirling's formula it is not hard to see that for any $\beta>0$ and $t\in\mathbb{R}$
$$
\lim_{N\to\infty}\mathbb{E} \exp\left(itN^{1/2}\left(\Wc^\mathrm{circular}-u(\beta)\right)\right)=\exp\left(-\frac{v(\beta)t^2}2\right)
$$
with 
$$
u(\beta)=\Psi\left(1+\frac\beta2\right)-\log\left(\frac\beta 
2\right),\qquad v(\beta)=\frac 2\beta -\Psi'\left(1+\frac\beta 2\right),\qquad \Psi(z)=\frac{\Gamma'(z)}{\Gamma(z)}\,,
$$
and the expression of $u(\beta)/2$ through the partition function $Z(\beta)$ coincides with that of the thermodynamic energy per particle. 
By L\'evy's continuity theorem this implies the central limit theorem
$$
\lim_{N\to\infty}\mathrm{Prob} \left\{\frac{\Wc^{\mathrm{circular}}-u(\beta)}{\sqrt{v(\beta)/N}}\le s\right\} =\frac 1{\sqrt{2\pi}}
\int_{-\infty}^s e^{-x^2/2}dx, \qquad s\in\mathbb{R}.
$$

It is natural to conjecture that the same central limit theorem holds for $\Wc$ on general $\beta$-sine processes
(constructed in \cite{VV}) as well, in particular 
providing asymptotic values $\lim\E\Wc=u(\beta)$ and $\lim N\cdot\mathrm{Var(\Wc)}=v(\beta)$ 
for any $\beta>0$. At the moment we do not know how to prove it,  although
in the $\beta=1,2,4$ cases one may hope to obtain a proof via controlling the asymptotics
of moments of $\Wc$ through explicit formulas for the correlation functions.

In the follow-up paper \cite{dysonmehta}, prompted by the wish to analyze experimental data, Dyson and Mehta 
looked for a way of approximating the energy per particle, in the case $\beta=1$, by averaging a statistic 
of the form $\sum F(a_i,a_j)$ over windows of increasing length in the random matrix spectrum. The concentration
requirement of the asymptotically vanishing variance led them to the conclusion that with their choice
of $F$, the statistic had to be corrected with further interaction terms \cite[Eq. (109)]{dysonmehta}. 
Our renormalized energy $\Wc$ seems to be a better solution to Dyson-Mehta's problem, although one
should note that Dyson-Mehta's statistic is purely local, while $\Wc$ involves long range
interaction (between leftmost and rightmost particles in $[0,N]$).

\medskip
   
 The paper is organized as follows:
 In Section \ref{secdefW} we give the precise definitions of $W$, 
 the context about it  and its connection to Coulomb energy, from \cite{ss1,ma1d}. This leads us to the  definition of $\Wc$. In Section \ref{sec3} we compute  limits of expectations of $\Wc$: First on the real line and for general processes, then 
for the $\beta=1,2,4$ sine processes, then in the plane and for general processes, finally for the explicit Ginibre and zeros of GAF processes. In Section \ref{optim} we find minimizers of $\lim \Wc$  
among determinantal processes with translation-invariant correlation kernels. In Section~\ref{sec4} we compute the limit variance of $\Wc$  for our specific processes (and show it is $0$).
In Section~\ref{sc:misc} we gather some miscellaneous computations: the effect on $\lim \E \Wc$   of superposition and decimation of processes and the computation of expectation of $\lim \E \Wc$ for the $\beta=2$ discrete sine processes.

\vskip .5cm
{\bf Acknowledgements:} We are very grateful to Peter Forrester for his idea of using circular ensembles for the
heuristics above, and also for drawing our attention to \cite{jancovici}.  We also thank the anonymous referee for a careful reading and very interesting comments.
A. B.  thanks  Alan Edelman for help with numerics  and S. S. thanks  C. Sinan G\"unt\"urk for helpful discussions related to Section \ref{optim}.
A. B. was  partially supported by NSF grant DMS-1056390, S. S. was supported by a EURYI award. We also thank the MSRI for support to attend a workshop where this work was initiated.

\section{Definitions of $W$ and $\Wc$} \label{secdefW} The aim of this section is to present a definition of  $W$   for a configuration of points  either on the line or in the plane, which is directly deduced from that of  \cite{ss1,ma1d}, but depends only on the data of the points. We do not attempt here to fully and  rigorously connect the next definition with that of \cite{ss1,ma1d} (since we believe it presents some serious technical difficulties) but the link will be readily apparent.

We start by recalling the definitions from \cite{ss1}, but rather in the form presented in \cite{ma2d}.
\subsection{Definition of $W(j)$ in the plane}

In \cite{ss1}, $W$ was
 associated to a discrete set of points (with asymptotic density) $m$ in
the plane, via a vector field $j$:  if  $j$  is a vector field in $\mr^2$ satisfying
\begin{equation}\label{eqj}
\curl j = 2\pi (\nu - m), \qquad \div j=0,\end{equation} where $m$ is a positive number (corresponding to the average point density) and  $\nu$
has the form
\begin{equation}\label{eqnnu}
\nu= \sum_{p \in\Lambda} \delta_{p}\quad \text{ for
some discrete set} \  \Lambda\subset\mr^2,\end{equation} then
 for any function $\chi$ we define
\begin{equation}\label{WR}W(j, \chi): = \lim_{\eta\to 0} \(
\hal\int_{\mr^2 \backslash \cup_{p\in\Lambda} B(p,\eta) }\chi |j|^2
+  \pi \log \eta \sum_{p\in\Lambda} \chi (p) \).
\end{equation}

\begin{defi}\label{defA} Let $m$ be a nonnegative number.
Let  $j$
be a vector field in $\mr^2$. We say $j$
belongs to the admissible class
$\mathcal{A}_m $ 
 if \eqref{eqj}, \eqref{eqnnu}  hold  
and
\begin{equation}\label{densbornee}
\frac{    \nu (B_R )  } {|B_R|}\quad \text{ is bounded by a constant independent of $R>1$}.
\end{equation}
 \end{defi}

\begin{defi}
The ``renormalized energy" $W(j)$ relative to the family of balls (centered at the origin) $\{B_R\}_{R>0}$  (and the number  $m$) in $\mr^2$ is  defined for $j \in \mathcal{A}_m$ by
\begin{equation}\label{WU} W(j):= \limsup_{R \to \infty}
\frac{W(j, \chi_{R})}{|B_R|} ,
\end{equation}
where
 $\chi_{R}$ denotes  positive  cutoff functions satisfying, for some
constant $C$ independent of $R$,
\begin{equation}
\label{defchi} |\nab \chi_{R}|\le C, \quad \supp(\chi_{R})
\subset B_R, \quad \chi_{R}(x)=1 \ \text{if } d(x, (B_R)^c) \ge
1.\end{equation}
\end{defi}

 Note that by scaling, the density of points can be changed to any, and the rule for that change of scales is
  \begin{equation}\label{chscale}
  W(j)= m \(W(j') - \frac{\pi}{2} \log m\),\end{equation}
where  $\nu$ has density $m$ and $j'= \frac{1}{\sqrt{m}}j (\frac{\cdot}{\sqrt{m}})$ (hence the set $\nu'$ has density $1$).

This function was first  introduced in \cite{ss1} and derived as a limiting interaction
energy for vortices of Ginzburg-Landau configurations.
Independently of Ginzburg-Landau it can be viewed
as a Coulombian interaction energy for an infinite number of points in the plane, computed via a renormalization. Many of its properties are stated in \cite{ss1}. In \cite{ma2d} it was directly connected to 2D log gases. We will give more details and properties of $W$ in Section~\ref{sec:background}.

\subsection{Definition of $W(j)$ on the line}
In \cite{ma1d} a one-dimensional analogue to the above definition is introduced for the study of  1D log gases, we now present it.
The renormalized energy between points  in 1D is obtained by ``embedding" the real line in the plane and computing the renormalized energy in the plane, as defined in \cite{ss1}.
More specifically, we introduce the following definitions:

$\mr$ denotes the set of real numbers, but also the real line of the plane $\mr^2$ i.e. points of the form $(x,0)\in \mr^2$.   For the sake of clarity, we denote points in $\mr$ by the letter $x$ and points in the plane by $z=(x,y)$. For a function $\chi $  on $\mr$, we define its natural extension $\bar{\chi}$ to  a function on $\mr^2$ by $\bar{\chi}(x,y)=\chi(x)$.
$I_R$  denotes the interval $[-R/2, R/2]$ in $\mr$.

$\dr$ denotes the measure of length on  $\mr$ seen as embedded in $\mr^2$, that is
$$\int_{\mr^2} \varphi  \dr  = \int_\mr \varphi(x, 0)\, dx$$
for any smooth compactly supported test function $\vp $ in $\mr^2$.
This measure can be multiplied by bounded functions on the real-line.

\begin{defi}For any function $\chi $ on $\mr$, and any function $h$ in $\mr^2$  such that
\begin{equation}\label{eqj1}
-\Delta h = 2\pi (\nu - m \dr) \qquad \text{in } \ \mr^2\end{equation} where
 $m$ is a nonnegative number  and $\nu$ has the form
\begin{equation} \label{eqnnu1}\nu=  \sum_{p \in\Lambda} \delta_{p}\quad \text{ for
some discrete set of points of  $\mr$} ,\end{equation}
 we denote $j=-\np h :=(\p_2 h, - \p_1 h)$ and
\begin{equation}\label{WR1}W(j, \chi): = \lim_{\eta\to 0} \(
\hal\int_{\mr^2 \backslash \cup_{p\in\Lambda} B(p,\eta) }\bar{\chi}
|j|^2 +  \pi \log \eta \sum_{p\in\Lambda} \bar{\chi} (p) \),
\end{equation}
where $\bar{\chi}$ is the natural extension of $\chi$.
\end{defi}

\begin{defi}\label{defA} Let $m $ be a nonnegative number.
 We say $j=-\np h$
belongs to the admissible class
$\mathcal{A}_m $ if \eqref{eqj1}, \eqref{eqnnu1} hold
and
\begin{equation}\label{densbornee}
\frac{    \nu (I_R )  } {R}\quad \text{ is bounded by a constant independent of $R$}.
\end{equation}
 \end{defi}
We  use the
notation $\chi_{R}$ for positive  cutoff functions over  $\mr$ satisfying, for some
constant $C$ independent of $R$,
\begin{equation}
\label{defchi1} |\nab \chi_{R}|\le C, \quad \supp(\chi_{R})
\subset I_R, \quad \chi_{R}(x)=1 \ \text{if } |x|<\frac{R}{2}-1 .\end{equation}

\begin{defi}The renormalized energy $W$ is defined,
for $j \in \mathcal{A}_m$,  by
\begin{equation} \label{Wroi1} W(j):= \limsup_{R \to \infty}
\frac{W(j, \chi_{R})}{R} .
\end{equation}
\end{defi}

Note that while $W$ in 2D can be viewed as a renormalized way of computing $\|h\|_{H^1(\mr^2)}$, in 1D it amounts rather to a renormalized computation of $\|h\|_{H^{1/2}(\mr)}$.

In one dimension, the formula for change of scales is
\begin{equation}\label{minalai1}
W(j)= m \(W(j') - \pi \log m\).\end{equation}
where $j$ corresponds to density $m$ and $j'$ to  $m=1$.

\subsection{Background}\label{sec:background}
We recall here some properties and background from \cite{ss1,ma1d}. 

\begin{itemize}

\item[-] Since in the neighborhood of $p\in\Lambda$ we have $\curl j = 2\pi\delta_p - m(x)$, $\div j = 0$, we have near $p$ the decomposition $j(x) = \nab^\perp \log|x-p| + f(x)$ where $f$ is locally bounded, and it easily follows that the limits \eqref{WR}, \eqref{WR1} exists.  It also follows that $j$ belongs to $L^p_\loc$ for any $p<2$.

\item[-] Because the number of points is infinite,
 the interaction over large balls needs to be normalized
 by the volume as in a thermodynamic limit, and thus  $W$ does not feel compact perturbations of the configuration of points. Even though the interactions are long-range, this is not difficult to justify rigorously.
\item[-] The  cut-off function $\chi_R$ cannot simply be replaced by the characteristic function of  $B_R$ because for every $p\in\Lambda$
$$\lim_{\substack{R\to|p|\\R<|p|}} W(j,\indic_{B_R}) = +\infty,\quad \lim_{\substack{R\to|p|\\R>|p|}} W(j,\indic_{B_R}) = -\infty.$$
\item[-]In dimension 1, there is a unique $h$ satisfying \eqref{eqj1} and for which $W(- \np h)$ is finite. Thus $W$ amounts to a function of $\nu$ only.
\item[-] In dimension 2, there is not a unique $j$ for which $W(j)<\infty$  and it would be tempting to define $W$ as a function of $\nu$ only and not $j$ by minimizing over $j$, however we do not know how to prove that the resulting object has good properties such as measurability.
\end{itemize}

An important question is that of  characterizing the minimum and minimizers of $W$. 
 For the case of dimension 1, it is proven in \cite{ma1d} that the value 
of $W$ does not depend on the choice of $\chi_R$ satisfying \eqref{defchi1}, that $W$ is Borel-measurable over  $L^p_{loc}(\mr^2, \mr^2)$ for $p<2$, and that 
\begin{equation}\label{minen1d}
\min_{\mathcal{A}_1} W = - \pi \log (2\pi)\end{equation} is achieved for $\nu=\sum_{p\in \mz} \delta_p$.

 For the case of dimension 2,
it is proven in Theorem 1 of \cite{ss1}  that the value  of $W$ does not depend on $\{\chi_{B_R}\}_R$ as long as it satisfies \eqref{defchi}, that $W$ is Borel-measurable over  $L^p_{loc}(\mr^2, \mr^2)$, for $p<2$ and that $\min_{\mathcal{A}_m} W$ 
is achieved and finite. Moreover, there exists a minimizing sequence consisting of periodic vector fields.
The value of  the minimum of $W$  is not known, it is conjectured (see \cite{ss1}) that it is asymptotically equal to the value at the perfect triangular lattice.  This conjecture is supported by the fact that the triangular lattice can be proved to be the minimizer among lattice configurations of fixed volume.

In addition to arising from the study of Ginzburg-Landau \cite{ss1}, 
$W$ is naturally connected, as shown in \cite{ma2d,ma1d}, to Vandermonde factors  such as  $e^{-\frac{\beta}{2} w_n}$ with
\begin{equation}
\label{wn}
w_n (x_1, \dots, x_n)= -  \sum_{i \neq j} \log |x_i-x_j| +n  \sum_{i=1}^n V(x_i),\end{equation} where $x_i, \dots , x_n \in \mr^d$ with $d=1 $ or $2$,
 for any potential $V$ (for example $V$ quadratic).
To explain this, let us introduce some more notation. We set
\begin{equation}\label{Ib}
I (\mu)= \int_{\mr^d\times \mr^d} - \log |x-y|\, d\mu(x) \, d\mu(y) + \int_{\mr^d} V(x)\, d\mu(x). \end{equation}  It is well known in potential theory (see \cite{safftotik}) that, provided $V(x)- \log |x| \to +\infty $  as $|x|\to \infty$,  $I$ has a unique minimizer among probability measures, called the equilibrium measure -- let us  denote it $\mo$. It  is characterized by the fact that
there exists a constant $c$ such that
\begin{equation}\label{optcondmo}
U^{\mo} + \frac{V}{2}= c\  \text{ in the support of $\mo$, and} \  U^\mo + \frac{V}{2}\ge c \text{ everywhere}\end{equation}
 where for any $\mu$, $U^\mu$ is the potential generated by $\mu$, defined by
\begin{equation}
U^\mu(x)= - \int_{\mr^d} \log |x-y|\, d\mu(y).\end{equation}We may then set
 $\z= U^\mo + \frac{V}{2} - c$ where $c$ is the constant in \eqref{optcondmo}. It satisfies
\begin{equation}\label{eqz}
\left\{\begin{array}{rlll}
 \z &=& 0 & \text{in $\supp(\mu_0)$}\\
\z&>& 0 & \text{in $\mr^d\sm \supp(\mu_0)$}\end{array}\right.\end{equation}
It is easy to check that $\z$ grows like $V$ at infinity.
The connection between $w_n$ and $W$ is given by the following exact ``splitting formula", valid in both 1 and 2 dimensions.
\begin{lem}[\cite{ma1d,ma2d}]\label{lem26}
For any $x_1, \dots, x_n\in \mr^d$,  $d=1$ or $d=2$,  the following holds
\begin{equation}\label{idwn}
w_n(x_1, \dots, x_n)= n^2 I (\mo)- \frac{n}{d} \log n + \frac{1}{\pi} W(-\np  H', \indic_{\mr^d} )  +2 n \sum_{i=1}^n  \z (x_i)\end{equation}
where  $W$ is defined in \eqref{WR} or \eqref{WR1} respectively,  and  where
\begin{equation}\label{2.2}
 H= - 2\pi\Delta^{-1} \(\sum_{i=1}^n \delta_{x_i} - n \mu_0 \) \ \text{in } \ \mr^2,\quad H'(n^{1/d}  x) = H(x).\end{equation} where the equation is solved in $\mr^2$ and
  $\mu_0$ is naturally extended  into a measure on $\mr^2$ if $d=1$. (Note that $\Delta^{-1}$ is the convolution with $1/(2\pi)\log |\cdot|$.)
\end{lem}
\begin{proof}[Sketch of the proof]
We start by writing
$$w_n(x_1, \dots, x_n) =\int_{\triangle^c}\,- \log |x-y| \, d\nu(x)\, d\nu(y) + n\int V(x)\, d\nu(x)$$ where $\triangle $ denotes the diagonal of $(\mr^d)^2$ and $\nu = \sum_{i=1}^n \delta_{x_i}$. The idea is to compute the right-hand side by  splitting $\nu$ as $n \mo + \nu - n\mo$. This way,  using the fact that $\mo(\triangle) = 0$, we obtain
\begin{multline*}
w_n(x_1, \dots, x_n) = n^2 I(\mo)+  2n \int U^{\mo}(x)\, d(\nu - n \mo)(x)+ n\int V(x)\, d(\nu- n \mo)(x)\\+
\int_{\triangle^c}\,- \log |x-y| \, d(\nu - n \mo)(x)\, d(\nu- n \mo)(y) .
\end{multline*}
Since $U^{\mo} + \frac{V}{2}= c + \z$ and since $\nu$ and $n \mo$ have same mass $n$, we have
$$2n \int U^{\mo}(x)\, d(\nu - n \mo)(x)+ n\int V(x)\, d(\nu- n \mo)(x)=  2 n \int \z \,d(\nu - \mo)= 2n \int \z \,d\nu,$$
where we used  that  $\z=0$ on  the support  of $\mo$. Thus 
\begin{equation}\label{avantchvar}w(x_1, \dots, x_n) = n^2 I(\mo)+  2n \int \z \,d\nu 
+
\int_{\triangle^c}\,- \log |x-y| \, d(\nu - n \mo)(x)\, d(\nu- n \mo)(y) .
\end{equation}
We then claim that the last term in the right-hand side is equal to 
$\frac{1}{\pi} W(-\np H ,\indic_{\mr^d})$.
We define
$H_i(x) := H(x)+\log|x-x_i|$. We have $H_i = -\log\ast (\nu_i - n\mo)$, with $\nu_i = \nu-\delta_{x_i}$, and near $x_i$, $H_i$ is $C^1$. It is then not difficult to deduce 
\begin{equation}\label{rhsss}\int_{\triangle^c}\,- \log |x-y| \, d(\nu - n \mo)(x)\, d(\nu- n \mo)(y)= \sum_i H_i(x_i) - n\int H(x)\,d\mo(x).\end{equation}
On the other hand, by definition \eqref{WR}, \eqref{WR1} and using Green's formula, we have 
$$\frac{1}{2\pi} \int_{\mr^2 \sm\cup_i B(x_i,\eta)} |\nab H|^2 =\frac{1}{2\pi} \sum_i \int_{\p B(x_i,\eta)} H\frac{\p H}{\p n} + \hal\int_{\mr^2 \sm \cup_i B(x_i,\eta)}H\,d(\nu - n\mo).$$ Using the decomposition $H= H_i - \log |x-x_i|$ near each $x_i$, adding $n \log \eta$ and then letting $\eta \to 0$, we arrive at the same result as the right-hand side of \eqref{rhsss}. This establishes the claim, and  the final result then follows from the change of scales $x' = n^{1/d} x$ in  $W$.

\end{proof}

\subsection{Calculation for points on a torus in dimension two}

The reason why we need another definition here is that, as we already mentioned, we wish to define $W$ based on the knowledge of $\nu$, i.e. a set of points, alone. For a given $\nu$, there is no uniqueness of the $j$ satisfying \eqref{eqj}  or \eqref{eqj1} (in fact the indetermination is exactly a constant, see  \cite[Lemma 1.4]{ma2d}), and this makes it problematic to define $W$ for points only in a measurable manner.
However, if the configuration of points is periodic, then $W$ can be computed nonambiguously from $\nu$ only.

The formula for this in dimension 2 is given in \cite{ss1}.
 If the periodicity  is
 with respect to a  torus $\T$, and the number of points in that torus is  $n$ (denote them $a_1, \dots, a_n$) then there exists a unique (up to a constant) solution to
 \begin{equation}\label{Ha}
 -\Delta H_{\{a_i\}}= 2\pi \(\sum_{i=1}^n \delta_{a_i
 } - \frac{n}{|\T|} \) \quad \text{on} \ \T \end{equation}
 (i.e. a periodic solution) and we let \begin{equation}\label{ja}
 j_{\{a_i\}}= - \np H_{\{a_i\}}.\end{equation}  It was proved in \cite{ss1} that $W(j_{\{a_i\}})$ is the smallest of the $W(j)$ over all $j$ satisfying \eqref{eqj} which are  $\T$-periodic, and it was established
 (see formula (1.10)) that
  $$W(j_{\{a_i\}})= \hal \sum_{i \neq j} G(a_i-  a_j) + n c$$
  where $c$ is a  constant, and $G$ is the Green's function associated to the torus in the following way:
  \begin{equation}\label{green2d}-\Delta G= 2\pi\( \delta_0  -\frac{1}{|\T|} \) \qquad \text{in} \ \T.\end{equation}
   In \cite{ss1}, (1.11), an explicit formula is given for this:
\begin{equation}\label{formexpl}
W(j_{\{a_i\}} )= \hal \sum_{i \neq j} \sum_{p \in (\mz^2)^* \sm \{0\}} \frac{e^{2i\pi p \cdot (a_i- a_j)}  }{4\pi^2 |p|^2 } + \frac{n}{2}
 \lim_{x\to 0} \( \sum_{p \in (\mz^2)^*\backslash
\{0\} } \frac{ e^{2i\pi p \cdot x}}{ 4\pi^2 |p|^2}  + \log
|x|\). \end{equation}
which can in turn be expressed using Eisenstein series.

For simplicity we prefer to work with density $1$ in a square torus, which is then necessarily 
 $\T_N:=\mr^2 /(N\mz)^2$ where  $n=N^2$. Compared to \cite{ss1}, this will change the constants in the formulae.
So for the sake of clarity we  will  include below a complete proof of the following:

\begin{lem}\label{Wnper}
Let $a_1, \dots , a_n$ be $n=N^2$ points on $\T_N$. Let $j_{\{a_i\}}$ be the $\T_N$-periodic vector field associated through \eqref{ja} to the configuration of
 points $\{a_i\}$, extended by periodicity to the whole plane. Then $W(j_{\{a_i\}})$ as defined in \eqref{WU} is a function of $a_1, \dots, a_n$ only, which is equal to
\begin{equation}\label{Wtor}
W_N(a_1, \dots, a_n)= \frac{\pi}{N^2} \sum_{i \neq j} G(a_i-a_j) + \pi R \end{equation}
where $G$ is the unique solution of
\begin{equation}\label{G2d}
 -\Delta G=  2\pi\( \delta_0  -\frac{1}{N^2}\) \qquad \text{in} \ \T_N,\end{equation}
 with $\dashint_{\T_N} G(x)\, dx=0$; and $R$ is a constant, equal to $\lim_{x\to 0} \( G(x)+ \log |x|\)$.
\end{lem}
\begin{proof} First, arguing as in \cite{ss1} Proposition  3.1, we have
\begin{equation}\label{f1} W(j_{\{a_i\}})= \frac{1}{|\T_N|}  \(\lim_{\eta \to 0} \hal \int_{\T_N \backslash \cup_{i=1}^n B(a_i, \eta)} |j_{\{a_i\}} |^2 +  \pi n \log \eta\),\end{equation} which means we reduce the computation in the plane to a computation on the torus.
 The next step is a renormalized computation \`a la \cite{bbh}, as skeched in the proof of Lemma \ref{lem26}.

 Solving for \eqref{Ha} via the Green function  \eqref{G2d}, in view of the translation invariance of the equation, we can choose
\begin{equation}\label{f2} H_{\{a_i\}}(x) = \sum_{i=1}^n G(x- a_i).\end{equation} Let us denote \begin{equation}\label{RG}
R(x)= G(x)+ \log |x|,\end{equation} it is well-known that   $R(x)$ is a continuous function.  We denote $R:=R(0)$.

Inserting \eqref{ja} into \eqref{f1}, we have
\begin{equation}\label{f3}
W(j_{\{a_i\}})= \frac{1}{|\T_N|}  \(\lim_{\eta \to 0} \hal \int_{\T_N \backslash \cup_{i=1}^n B(a_i, \eta)} |\nab H_{\{a_i\}} |^2 +  \pi n \log \eta\).\end{equation}
Using  Green's formula for integration by parts,  we have
$$\int_{\T_N \backslash \cup_{i=1}^n B(a_i, \eta)} |\nab H_{\{a_i\}} |^2= \int_{\T_N \backslash \cup_{i=1}^n B(a_i, \eta)} (- \Delta  H_{\{a_i\}} )  H_{\{a_i\}}  + \sum_{i=1}^n \int_{\p B(a_i, \eta)}   H_{\{a_i\}}  \frac{\p H_{\{a_i\}}  }{\p  n} .$$
Inserting  \eqref{Ha} and \eqref{f2}, we have
\begin{multline*}\int_{\T_N \backslash \cup_{i=1}^n B(a_i, \eta)} |\nab H_{\{a_i\}} |^2 = - \sum_{j=1}^n \int_{\T_N \backslash \cup_{i=1}^n B(a_i, \eta)} 2\pi G(x-a_j) \, dx \\+ \sum_{i =1}^n \sum_{j=1}^n \int_{\p B(a_i, \eta)} G(x-a_j) \frac{\p H_{\{a_i\}}  }{\p n} , \end{multline*} with $n$ the inner unit normal to the circles.
Using for the first term the fact that $G$ is chosen to have average zero, and   for the second term splitting $G(x- a_j)$ as $ -\log |x-a_j|+ R(x-a_j)$, using the continuity of $R(x)$ and $\lim_{\eta \to 0} \int_{\p B(a_i, \eta)} \frac{\p H_{\{a_i\}}  }{\p \nu}=2\pi $ (from \eqref{Ha}),  and letting $\eta \to 0$ we find
$$\lim_{\eta\to 0} \hal \int_{\T_N \backslash \cup_{i=1}^n B(a_i, \eta)} |\nab H_{\{a_i\}} |^2 + \pi n \log \eta=   \pi \sum_{i =1}^n  R(0) + \pi \sum_{i \neq j}  (- \log |a_i-a_j| + R(a_i-a_j)).$$
Inserting into \eqref{f3} yields   \eqref{Wtor}.
\end{proof}
It turns out that $G$ can be expressed through  an Eisenstein series.  The Eisenstein series with parameters  $\tau \in \mc$ and $u,v \in \mr$  (cf.  \cite{lang}) is defined by
\begin{equation}
\label{eis}
E_{u,v}(\tau):= \sum_{(m,n)\in \mz^2 \sm \{0\}} e^{2i\pi (mu+nv)}\frac{  Im (\tau) }{|m\tau +n|^2}.
\end{equation}
Let us now define, for $x \in \mc$,
\begin{equation}\label{defE}
E_N(x):= E_{\Re(x/N), \Im (x/N)} (i).\end{equation}

As in \cite{ss1} we will also need another classical modular function: the Dedekind $\eta $ function. It is defined,  for $\tau$ in the upper-half of the complex plane  by
\begin{equation}\label{defeta}
\eta(\tau)= q^{1/24} \prod_{k=1}^\infty (1-q^k)\qquad \text{where} \ q=e^{2i\pi \tau}.\end{equation}

The following holds

\begin{pro}\label{form2d}
Let $a_1, \dots , a_n$ be $n=N^2$ points on $\T_N$, and let $W_N$ be defined as in
\eqref{Wtor}. $E_N$ being defined in \eqref{defE}, we have
\begin{equation}\label{WNE}
W_N (a_1, \dots , a_n)=\frac{1}{2N^2} \sum_{i\neq j} E_N(a_i-a_j) + \pi \log \frac{N}{2\pi} - 2\pi \log \eta(i),\end{equation} where points in the plane are identified with complex numbers.
\end{pro}
\begin{proof}
As in \cite{ss1}, $G$, introduced in Lemma \ref{Wnper},  can be computed from \eqref{G2d} using Fourier series and this yields:
\begin{equation}\label{Gser}
G(x)= \sum_{p\in \mz^2 \sm \{0\}} \frac{e^{\frac{2i\pi}{N} p \cdot x}}{2\pi |p|^2}.\end{equation}
We recognize here an Eisenstein series:
\begin{equation}\label{Geis}
G(x)=\frac{1}{2\pi} E_{\Re(\frac{x}{N}), \Im (\frac{x}{N}) } (i) =\frac{1}{2\pi} E_N( x) .\end{equation}
Note that the fact that $G$ has zero average implies that
\begin{equation}\label{zeroavE}
\int_{\T_N} E_N(x)\, dx=0.\end{equation}
There remains to compute the value of the constant $R= \lim_{x\to 0} \( G(x)+ \log |x|\)$. For that we use the ``second Kronecker limit formula" (see \cite{lang}) that asserts that
$$E_{u,v}(\tau)=-2\pi \log |f(u-v\tau, \tau)q^{v^2/2}|, $$
where $q=e^{2i\pi \tau}$, $p=e^{2i\pi z}$, $z=u-v\tau$, and
\begin{equation}\label{f}
f(z,\tau)=q^{1/12} (p^{1/2}-p^{-1/2}) \prod_{k\ge 1}(1-q^k p)(1-q^k/p).\end{equation}

We use this formula with $\tau=i$, $u= \Re(\frac{x}{N})$, $v=\Im (\frac{x}{N})$. In that case $q=e^{-2\pi}$, $z=\frac{\overline{x}}{N}$ and $p=e^{2i\pi \frac{\overline{x}}{N} }
$ (where the bar denotes complex conjugation), and the formula yields
\begin{equation}\label{klf}
E_N(x)= -2\pi \log \left|f\(\frac{\overline{x}}{N}   , i\) e^{- \pi (\Im \frac{x}{N})^2  }\right| .
\end{equation}
As $x\to 0$ we have $p\to 1$ and
$$p^{1/2}-p^{-1/2} \sim 2i \pi \frac{\overline{x}}{N} $$
and also (see \eqref{defeta})
$$ q^{1/12}\prod_{k\ge 1} (1-q^k p)(1-q^k/p) \sim_{x \to 0} q^{1/12} \prod_{k\ge 1}
(1-q^k)^2 = \eta(i)^2  . $$
So as $x \to 0$, we have
\begin{equation*}
E_N(x) \sim -2\pi \log \left|\eta(i)^2 2i\pi \frac{\overline{x}}{N}\right|
= -2\pi \log \( \frac{2\pi |x|}{N}\) - 4\pi \log \eta(i).\end{equation*}
Combining with \eqref{Geis}, it follows that $R= - \log \frac{2\pi}{N} - 2\log \eta(i).$ Inserting this and \eqref{Geis} into \eqref{Wtor} we get the result.
 \end{proof}

We emphasize here that this formula is the exact value for $W$ as defined in \eqref{WU} provided we assume full periodicity with respect to $\T_N$.

\subsection{Calculation for points on a torus in dimension  one}
We do here the analogue in dimension 1, i.e.  we compute $W$ given by \eqref{Wroi1} assuming that the point configuration is periodic with respect to $\T_N:=\mr/(N\mz)$. We assume that there are $n=N$ points  $a_1, \dots, a_N$ in $\T_N$, hence $W$ is computed with respect to  the average density 1, i.e. $m= 2\pi$. There exists a unique (up to a constant)  $H_{\{a_i\}}$ satisfying
$$-\Delta H_{\{a_i\}}= 2\pi \(\sum_{i=1}^n \delta_{(a_i,0)}-\dr\)$$ in $\mr^2$ and  which is $N$-periodic  with respect to the first variable, i.e $H_{\{a_i\}}(x+ N, y)=H_{\{a_i\}} (x, y)$ for every $(x, y) \in \mr^2$. We then set $j_{\{a_i\}}= - \np H_{\{a_i\}}$.
Then, as in \cite{ma1d} we have
\begin{lem}
Let $a_1, \dots , a_N$ be $N$ points on $\T_N=\mr/(N \mz)$. Let $j_{\{a_i\}}$ be as above.
Then $W(j_{\{a_i\}})$ as defined in \eqref{Wroi1} is a function of $a_1, \dots, a_N$ only, which is equal to
\begin{equation}\label{Wtor1d}
W_N(a_1, \dots, a_N)= \frac{\pi}{N} \sum_{i \neq j} G(a_i-a_j, 0) +\pi R \end{equation}
where
$G(z)$ is a restriction of the  Green function of $\T_N \times \mr $, more precisely the   solution to
\begin{equation}\label{G3}
- \Delta_z G(z)=2\pi \(\delta_{(0,0)} - \frac{1}{N} \delta_{\T_N}\)  \qquad z\in \T_N \times \mr \end{equation} with $\dashint_{\T_N} G(x,0)\, dx=0$, and $R$ is a constant, equal to $$\lim_{x\to 0}  \( G(x,0) + \log |x|\).$$ Here $\delta_{\T_N} $ is the distribution defined  over $\T_N\times \mr$ by $\int \delta_{\T_N} \phi= \int_{\T_N} \phi(x,0)\, dx.$

\end{lem}
\begin{proof}The proof is analogous to that of Lemma \ref{Wnper}.
First we have
\begin{equation}\label{f5}
 W(j_{\{a_i\}})= \frac{1}{|\T_N|}  \(\lim_{\eta \to 0} \hal \int_{(\T_N\times \mr ) \backslash \cup_{i=1}^N B((a_i,0), \eta)} |\nab H_{\{a_i\}} |^2 +  \pi N \log \eta\).\end{equation}
 We again write
\begin{equation}\label{f6} H_{\{a_i\}}(z) = \sum_{i=1}^n G(z- (a_i,0) ), \end{equation} with  $G(z)= -\log |z|+R(z)$, $R$ a continuous function.
Using Green's formula   we have
\begin{multline*}
\int_{(\T_N\times \mr )\backslash \cup_{i=1}^N B((a_i,0) , \eta)} |\nab H_{\{a_i\}} |^2= -2  \int_{(\T_N\times \mr ) \backslash \cup_{i=1}^N  B((a_i,0), \eta)}    H_{\{a_i\}} \delta_{\T_N}  \\+ \sum_{i=1}^N  \int_{\p B((a_i,0) , \eta)}   H_{\{a_i\}}  \frac{\p H_{\{a_i\}}  }{\p n} .\end{multline*}
From
 \eqref{f6} we have $$\lim_{\eta\to 0}  \int_{(\T_N\times \mr ) \backslash \cup_{i=1}^N  B((a_i,0), \eta)}    H_{\{a_i\}} \delta_{\T_N} = \sum_{j=1}^N \int_{\T_N} G(x-a_j,0) \, dx  =0$$ by choice of $G$.
The rest of the proof is exactly as in Lemma \ref{Wnper}, inserting the splitting of $G$  and \eqref{G3}.
\end{proof}
On the other hand we shall see below that we have the following explicit formula for $G$ restricted to the real axis:
\begin{equation}
\label{Glogsin}G(x,0)= - \log \left|2\sin \frac{\pi x}{N}\right|.
\end{equation}
Note that a consequence of this and the zero average condition on $G$ is that
\begin{equation}\label{log0}
\int_{\T_N}  \log \left|2\sin \frac{\pi v}{N}\right|\, dv=0.\end{equation}
The previous lemma and \eqref{Glogsin}  lead us to the following
\begin{pro}
Let $a_1, \dots, a_N$ be $N$ points on $\T_N= \mr/(N\mz)$ and let $W_N$ be defined as in \eqref{Wtor1d}. We have
\begin{equation}\label{WNlog}
W_N(a_1, \dots, a_N)= -\frac{\pi}{N} \sum_{i \neq j} \log \left|2\sin \frac{\pi(a_i - a_j)}{N} \right|- \pi \log \frac{2\pi}{N}.\end{equation}
\end{pro}
\begin{proof} The first step is to prove \eqref{Glogsin}.
This is done by solving  \eqref{G3} in Fourier series/transform. We choose the following normalization for Fourier transforms and series:
$$\widehat{f}(\xi)= \int f(x) e^{-2i\pi x \cdot \xi }\, dx$$
$$c_k(f)= \int_{\T_N} f(x) e^{-\frac{2i\pi kx}{N}} \, dx.$$
Then  the Fourier inversion formula is $f(x) =\int \widehat{f}(\xi) e^{2i\pi x\cdot \xi} \, d\xi$ and respectively
$$f(x)= \frac{1}{N} \sum_{k \in \mz} c_k(f) e^{\frac{2i\pi k}{N} x} .$$
 Since $G$  is $\T_N$-periodic   in its first variable, we may define its Fourier transform as a Fourier series in the first variable and  a  Fourier  transform in the second, i.e., for $m \in \mz$ and $\xi \in \mr$, 
$$\widehat{G}(m , \xi)= \int_{x\in \T_N}  \int_{y \in \mr} G(x,y) e^{-\frac{2i \pi mx}{N}} e^{-2i \pi y \xi} \, dx\, dy.
$$
If $G$ solves \eqref{G3} then by Fourier transform, $\widehat{G}$ has to satisfy
$$4\pi^2 \(\frac{m^2}{N^2}+\xi^2\)\widehat{G}(m , \xi)= 2\pi \widehat{\delta_{(0,0)}} - \frac{2\pi}{N} \widehat{\delta_{\T_N}}. $$
It is direct to establish that
 $\widehat{\delta_{\T_N}} =N \delta_m^0$ with $\delta_m^0$ by definition equal to $1$ if $m =0$ and $0$ otherwise.
Combining these facts,
we obtain
$$ \widehat{G}(m,\xi) = \frac{1 - \delta_m^0}{2\pi \( \frac{m^2}{N^2}+\xi^2\)}
\qquad for\ (m, \xi) \neq (0,0).$$
The undetermination of $\widehat{G}$ at $(0,0)$ corresponds to the fact that $G$ is only determined by \eqref{G3} up to a constant.
By Fourier inversion, it follows that
\begin{multline*}
G(x,y)= \frac{1}{N} \sum_{m \in \mz} \int_{\xi \in \mr}
 \frac{1 - \delta_m^0   } {2\pi ( \frac{m^2}{N^2}+\xi^2   ) }   e^{2i\pi \frac{m}{N} x  + 2i \pi y \xi}  \, d\xi+c
= \frac{1}{N} \sum_{m \in  \mz^*   }   \int_{ \mr}
\frac{e^{2i\pi \frac{m}{N} x + 2i \pi y\xi }} {2\pi ( \frac{m^2}{N^2}+\xi^2)}  \, d\xi+c.
\end{multline*}
Using the formula $\int_0^\infty \frac{\cos (bx)}{x^2+z^2} \, dx= \frac{\pi  e^{-|b|z}  }{2z}$  (cf. \cite{pbm1}) with $b=2\pi y$ and $z=|m|/N$,  we arrive at
\begin{equation}
\label{Gcos}
G(x,y)= \sum_{m=1}^\infty \frac{\cos (2\pi\frac{ m}{N}x) }{m} e^{-2\pi |y|\frac{|m|}{N} }+c.\end{equation}
We next particularize to the case $y=0$, and use   Clausen's formula (cf. \cite[Chap. 4]{lewin})
\begin{equation}\label{clausen}
\sum_{k=1}^\infty\frac{\cos (kx)}{k} = - \log \left|2\sin \frac{x}{2}\right|  \quad \text{for} \  0<x<2\pi,
\end{equation} and thus 
we find $$G(x,0)= - \log \left|2\sin \frac{\pi x}{N}\right|+ c.$$
The constant $c$ should  be chosen so that $\dashint_{\T_N} G(x,0)\, dx=0$, which imposes  
$c=0$, in view of  the relation
 $\int_{\T_N} \log \left|2\sin \frac{\pi x}{N} \right|\, dx =0$, a direct consequence of \eqref{clausen}.
This establishes \eqref{Glogsin}. In addition the value of $R$ follows: $R= \lim_{x\to 0} G(x,0)+ \log |x|= - \log \frac{2\pi}{N}$ and inserting into \eqref{Wtor1d}, the result follows.
\end{proof}

\subsection{New definition   for a general point configuration in the plane/line}
In order to define $W$ for a general configuration  of points, without referring to the current they generate, we choose to base ourselves on the formulas \eqref{WNE} and \eqref{WNlog} of the previous subsection.

To get rid of some constants, we use  however  a different normalization\footnote{the choice of  normalization in dimension 1 is based on \eqref{minen1d}} and, for a given family of points $\{a_i\}$, we  define
\begin{equation}
\label{wc1d2}\Wc = - \frac{1}{N} \sum_{i\neq j, a_i, a_j \in  [0, N]} \log \left|2\sin \frac{\pi(a_i-a_j)} N\right|  + \log N\qquad \text{in dimension 1}\end{equation} respectively
\begin{equation}\label{wc2d2}
\Wc = \frac{1}{2\pi N^2 }  \sum_{i\neq j, a_i, a_j \in [0,N]^2} E_N(a_i-a_j) + \log \frac{N}{2\pi \eta(i)^2}\qquad \text{in dimension 2}.\end{equation}

The results of the previous sections suggest trying to define $\mathcal{W}$  as the limit as $N\to \infty$ of these quantities.
More precisely, for a given random point process, $\Wc$ becomes a random variable, and we try to see whether it has a limit as $N \to \infty$.
Again, we do not claim a complete rigorous connection between such a quantity (which only depends on the point configuration) and the original $W$, which we recall, was defined via a vector field $j$. We also emphasize that in \eqref{wc1d2} and \eqref{wc2d2} the number of points in $[0,N]$ resp. $[0,N]^2$ is no longer necessarily equal to $N$ resp. $N^2$.

Let us comment a little on the minimization of $\Wc$. Since our definition has been relaxed, the question of whether $\Wc$ achieves a minimum becomes unclear. However, in dimension 1, 
observing that still $\Wc(\mz)=0$,  we have  the following statement
(in  dimension 2 no such result is available):
\begin{lem}
Assume $a_1, \dots, a_k $ are $k$ points in $[0,N]$. Then
\begin{equation}\label{eq:new1}
  - \frac{1}{N} \sum_{i\neq j} \log \left|2\sin \frac{\pi(a_i-a_j)} N\right|  + \log N \ge   \( 1- \frac{k}{N}\) \log N + \frac{k}{N} \log \frac{N}{k}.
\end{equation} 
Thus, for any point configuration $\Wc\ge \left( 1- \frac{k}{N}\right) \log N + \frac{k}{N} \log \frac{N}{k} $, 
where $k$ is the number of points in $[0,N]$.
\end{lem}
\begin{proof}
The proof is a simple  adaptation of the proof of minimality of the perfect lattice in \cite{ma1d}. Let  $a_1,\dots, a_k\in [0,N]$, and assume $a_1<\dots <a_k$.
 Let us also denote $u_{1,i}= a_{i+1}-a_i $, with the convention $a_{ k+1}=a_1+ N $.  We have $\sum_{i=1}^k u_{1,i}=   N$.
 Similarly, let
$u_{p,i}= a_{i+p}-a_{i}$, with the convention $a_{k+l}=a_l+N$.  We have $\sum_{i=1}^k u_{p,i}= p N$.
By periodicity of $\sin$, we may view the points $a_i$ as living on the circle $\mr/(N\mz)$. When adding the terms in $a_i-a_j$ in the sum of \eqref{eq:new1}, we can split it according to the difference of $p=j-i$ but modulo $N$.  This way, there remains
\begin{equation}\label{3.12}
 - \frac{1}{N} \sum_{i\neq j} \log \left|2\sin \frac{\pi(a_i-a_j)} N\right|  + \log N=-
\frac{2}{N}\sum_{p=1}^{[k/2]} \sum_{i=1}^k   \log \left| 2\sin\frac{ \pi u_{p,i}}{N} \right|+ \log N,\end{equation} where $[\cdot ]$ denotes the integer part.
But the function $\log |2\sin x|$ is stricly concave on $[0,\pi]$. It follows that
$$\frac{1}{k}\sum_{i=1}^k \log \left| 2\sin\frac{ \pi u_{p,i}}{N} \right|\le \log \left|2\sin \( \frac{\pi }{N k} \sum_{i=1}^k  u_{p,i} \) \right|= \log \left|2\sin\frac{ p\pi}{k}\right|.$$
Inserting into \eqref{3.12} we obtain
\begin{equation}
\label{3.2}
- \frac{1}{N} \sum_{i\neq j} \log \left|2\sin \frac{\pi(a_i-a_j)} N\right|  + \log N\ge
   - \frac{2k}{N}          \sum_{p=1}^{[k/2]} \log \left|2\sin\frac{ p\pi}{k}\right|+ \log N.
\end{equation}
On the other hand, we know that $\Wc(\mz)=0$ which means that   $- \frac{2}{N} \sum_{p=1}^{[ N/2]}N \log \left|2\sin \frac{p \pi} N\right|  + \log N= 0$, but also, since this is true for arbitrary integers $N$,
\begin{equation}\label{wdez}
  - 2 \sum_{p=1}^{[k/2]} \log \left|2\sin \frac{p \pi} k\right|  + \log k=0 .\end{equation}
Inserting into \eqref{3.2} we are led to
\begin{equation}
\label{3.3}
- \frac{1}{N} \sum_{i\neq j} \log \left|2\sin \frac{\pi(a_i-a_j)} N\right|  + \log N\ge
 - \frac{k}{N}     \log k + \log N= \( 1- \frac{k}{N}\) \log N + \frac{k}{N} \log \frac{N}{k}.
 \end{equation}
\end{proof}

\section{Expectation of $\Wc$}\label{sec3}

\subsection{Expectation and 2-point correlation functions}\label{correl}
We now turn to evaluating $\Wc$ for  random point processes.
In view of its form \eqref{wc1d2}--\eqref{wc2d2},  the expectation
of the random variable $\Wc$ can  be computed from the sole knowledge of the second correlation function of the process.
  Indeed, recall that for any $k\ge 1$, the $k$-point correlation function $ \ro_k$ of a random point process in $\mr^d$ is 
characterized by the property that
\begin{equation}\label{corr}\E \sum_{i_1, \dots, i_k \ pairwise\  distinct} F(x_{i_1},
 \dots, x_{i_k})= \int_{(\mr^d)^k} F(x_{1}, \dots, x_{k}) \ro_k(x_{1}, \dots, x_{k})
\, dx_{1} \dots \, dx_{k} ,\end{equation}
where the expectation is with respect to our measure on locally finite subsets 
$X=\{x_j\}\subset \mathbb{R}^d$, and $F$ ranges over a suitable space of test functions, 
see e.g. \cite{dvj}.

We note here that determinantal processes are a particular class of processes characterized by the fact that  
the correlation functions can be expressed as
\begin{equation}\label{rodet}
 \ro_k(x_1,\dots, x_n)= \det \(K(x_i, x_j)\)_{i,j \in [1, k]}
\end{equation}
for some kernel $K(x,y)$, see \cite{Sos00}, \cite{Lyo03}, \cite{Joh05}, \cite{Kon05},
\cite{Hou06}, \cite{Sos06}, \cite{Bor11} and references therein.  This will be used later.

 Here we need specifically the two-point correlation function.
 In addition, for our function the formula simplifies when the process
is assumed to be stationary (i.e. translation invariant). 
From now on we make the basic 
assumption that we are dealing with a translation invariant multiplicity-free random 
point process in $\mr $ or $\mr^2$, of density $1$ (i.e $\ro_1\equiv 1$) with
absolutely continuous correlation functions (hence, \eqref{corr} holds).
If $\ro_2(x,y)$ is the two-point correlation function of such a process, it is of the 
form $r_2(x-y)$ since the process is stationary. It is more convenient to work  
with the ``second cluster function" $T_2= 1- r_2$ (we will give a general definition of
the cluster functions in Section \ref{sec4}).

Our basic assumptions thus imply
\begin{equation}\label{basicass}
\ro_1 \equiv 1 , \quad \ro_2(x,y)=1-T_2(x-y) \ \text{for some function } T_2.\end{equation}

By definition of $\ro_2$, the expectation of $\Wc$ is, in dimension 1, 
\begin{equation}\label{ew1d}\E \Wc= \frac{1}{N} \int_{[0,N]^2} \log \left|2 \sin \frac{\pi (x-y)}{N} \right|T_2(x-y)\, dx\, dy +\log N \end{equation} (where we have used \eqref{wc1d2} and \eqref{log0}) and
 respectively in dimension 2
 \begin{equation}\label{ew2d}\E \Wc =- \frac{1}{2\pi N^2} \int_{[0,N]^2 \times [0, N]^2} E_N(x-y) T_2(x-y) \, dx \, dy +  \log \frac{N}{2\pi \eta(i)^2} \end{equation} (where we have used \eqref{wc2d2} and \eqref{zeroavE}).
  The question is then whether these quantities have a limit as $N \to \infty$. As we show below, this will only be 
  true under additional assumptions which in particular ensure a sufficient decay of $T_2$.

Finally, it would be most interesting to find natural conditions on the random points behavior
(their spacing etc) that would guarantee the existence of a limit to $ \E\Wc$.

\subsection{Expectation for one-dimensional processes: theoretical formula}
\begin{theo}\label{th1}
  Consider a random point process  $\X$ on the real line with 
the one-point correlation function $\ro_1(x) $ and a two-point correlation 
function $\ro_2(x,y)$, satisfying \eqref{basicass}. Under the following assumptions
  \begin{enumerate}
   \item[1)] $\sup_{v\in \mr} |T_2(v)|<\infty$;
\item[2)] there exist a sequence $\{\a_N\}_{N\ge 1}$ such that $\log N \ll \a_N \ll N^{1/2 -\ep}$ 
as $N \to \infty$ (for some $\ep >0$) and uniformly in $ A\in [\a_N, N-\a_N]$ we have
 $$\int_{\a_N}^A T_2(v) \log \left|2\sin  \frac{\pi v}{N}\right|\, dv
 =o(1)\quad \text{as  }N\to\infty;$$ \end{enumerate}
 the following holds:
 \\
 - if $\D\int_{-\infty}^\infty T_2(v)\log |v| \, dv<\infty$ and $ \D \int_{-\infty}^\infty T_2(v)\, dx = c \neq 1 $, then   $\E \Wc \to \infty $ as $N\to \infty$;\footnote{We  believe that $\Wc$ should be bounded below or at least that the value $-\infty$ is in fact not taken.}\\
 - if $\D \int_{-\infty}^\infty T_2(v)\, dx =1$  and $1-\D\int_{-\a_N}^{\a_N} T_2(v)\, dv=o\bigl((\log N)^{-1}\bigr)$ for $\{\a_N\}_{N\ge 1}$ as above,
   then  $\lim_{N\to \infty} \E \Wc $ exists and is finite if and only if   $\int_{-\infty}^\infty T_2(v)\log |v| \, dv$ converges, and if so then
 \begin{equation}\label{resth1}
 \lim_{N\to \infty} \E \Wc =  \log 2\pi +  \int_{-\infty}^\infty \log | v| T_2(v) \, dv.\end{equation}
\end{theo}
\begin{remark} \label{2.1}
Condition 2) is satisfied by the  stronger one:
\begin{equation}\label{rem2.1}\int_B^\infty |T_2(v)|\, dv= o\( \frac{1}{\log B}\) \quad \text{as} \ B \to +\infty.\end{equation} To see this  it suffices to observe that on $[\a_N, N-\a_N]$ we have $|\sin\frac{\pi v}{N}|\ge|\sin \frac{\pi \a_N}{N}|$.
\end{remark}
\begin{proof}[Proof of Theorem \ref{th1}]
In view of  \eqref{ew1d} we need to compute
\begin{equation}\label{tocomp}
\lim_{N\to \infty}\frac{1}{N} \int_{[0,N]^2} \log \left|2\sin \frac{\pi (x-y)}{N}\right| T_2(x-y)\, dx\, dy + \log N.\end{equation}
 For $(x,y) \in [0,N]^2$ we denote $u=x+y$ and $v=x-y$.
 We then split $[0,N]^2$ into the disjoint union of the following domains, see Figure \ref{fig1},  where $\a_N$ is  in condition 2):
 \begin{eqnarray*}
& \cdot  &  D_0= \{(x, y) \in [0,N]^2, |v| \le \a_N,    \a_N \le u \le 2N- \a_N \}\\
& \cdot & D_{1} = \{(x, y)\in [0,N]^2, v\ge  N-  \a_N \} \\
& \cdot  & D_{1'}=\{ (x, y)\in [0,N]^2, v\le  -N+ \a_N \} \\
& \cdot & D_{2} = \{(x,y) \in [0,N ]^2 , u \ge 2N - \a_N\}\\
&  \cdot & D_{2'} = \{(x,y) \in [0,N]^2, u \le \a_N\}\\
& \cdot & D_3=  \{(x,y) \in [0,N ]^2 ,   \a_N \le v \le N-\a_N\}\\
& \cdot  & D_{3'} = \{(x,y) \in [0,N ]^2 , - N + \a_N  \le  v \le -\a_N\}\end{eqnarray*}
\begin{figure}\label{fig1}
\begin{center}
\includegraphics{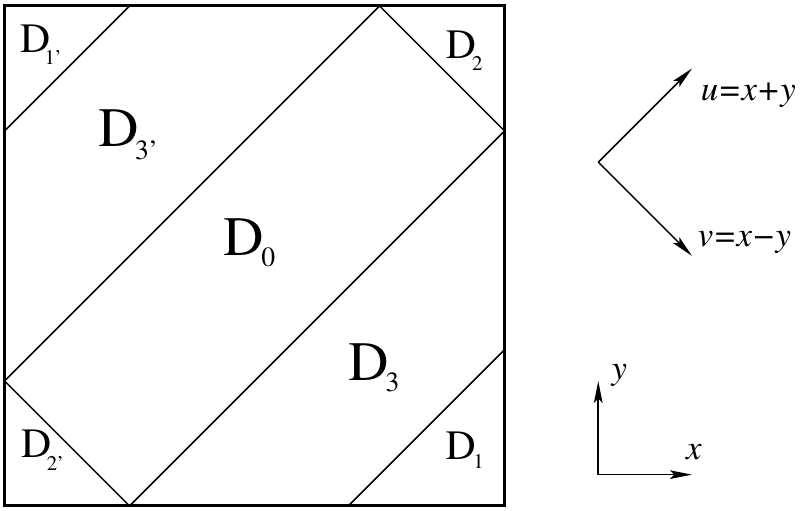}\caption{Splitting of the domain of integration}
\end{center}\end{figure}

We evaluate the integral in \eqref{tocomp} over each of these domains successively.
We start with the contribution of $D_1'$.
Making the change of variables $a=\frac{x}{N}, b=\frac{N-y}{N}$ we have
\begin{multline*}
 \int_{D_1'} \log \left|2\sin \frac{\pi (x-y)}{N}\right| T_2(x-y)\, dx\, dy
\\=  N^2 \int_{a\ge 0, b\ge 0, a+b \le\frac{\a_N}{N}} \log |2\sin \pi (a+b)|T_2(N(a+b-1))\, da\, db.\end{multline*}
Using assumption 1) and noting that in $D_1$, $\sin \pi (a+b)=\pi (a+b)+O((a+b)^3)$, and thus
$|\log |2\sin \pi (a+b)||= |\log |2\pi (a+b)||+ O(\frac{\a_N^2}{N^2})$,
we deduce
\begin{multline*}
\left|\int_{D_1'}  \log \left|2\sin \frac{\pi (x-y)}{N}\right| T_2(x-y)\, dx\, dy \right|\\
\le N^2 \sup |T_2| \int_{a\ge 0, b\ge 0, a+b \le\frac{ \a_N}{N}  } |\log |2\pi (a+b)||\, da \, db +O\(\frac{\a_N^4}{N^2}\).\end{multline*}
Using $\int_0^r \log  s \, ds= O( s |\log s|)$, it follows that
\begin{multline}\label{intd1}
\left|\int_{D_1'}  \log \left|2\sin \frac{\pi (x-y)}{N}\right| T_2(x-y)\, dx\, dy \right| \\ \le  C N^2  \sup |T_2|\frac{\a_N^2}{N^2} \left|\log \frac{\a_N}{N}\right|   +O\(\frac{\a_N^4}{N^2}\) =o(N)\end{multline} using $\alpha_N\ll N^{1/2-\ep}.$
The estimate for the domains  $D_{1}, D_{2}, D_{2'}$ is similar.
For the contribution over $D_3$, using the change of variables $(x, y) \to (u,v)$ we have
$$ \int_{D_3}  \log \left|2\sin \frac{\pi (x-y)}{N}\right| T_2(x-y)\, dx\, dy =
\hal (2N-2\a_N) \int_{\a_N}^{N-\a_N} T_2(v)\log \left|2\sin \frac{\pi v}{N}\right|\, dv.$$
This is $o(N)$ when assumption 2) holds. The estimate on $D_{3'}$ is completely analogous.
We have thus found that all contributions over $D_{1}, D_{1'}, D_2, D_{2'}, D_{3}, D_{3'}$ are negligible. The behavior of the integral will thus be determined by the contribution of $D_0$.

Changing  again the variables $(x,y) $ into $(u,v)$,
 we have
$$ \int_{D_0} \log \left|2\sin \frac{\pi (x-y)}{N}\right| T_2(x-y)\, dx\, dy =
\hal (2N-2\a_N)\int_{-\a_N}^{\a_N} T_2(v) \log \left|2\sin \frac{\pi v}{N}\right|\, dv.$$
But in $D_0$ we have $\sin \frac{\pi v}{N}= \frac{\pi v}{N}(1+ O(\frac{\a_N^2}{N^2}))$ as $N \to \infty$, hence $\log |\sin \frac{\pi v}{N}|= \log |\frac{ 2\pi v }{N}|+ O(\frac{\a_N^2}{N^2}).$
Therefore,
\begin{multline*}
\int_{D_0} \log \left|2\sin \frac{\pi (x-y)}{N}\right| T_2(x-y)\, dx\, dy \\ = (N -\a_N) \Big( - \log N \int_{-\a_N}^{\a_N} T_2(v)\, dv 
+ \int_{-\a_N}^{\a_N} \log |2\pi v| T_2(v) \, dv \Big) + \sup|T_2| O\(\frac{\a_N^3}{N}\).\end{multline*}
Using $\a_N \ll N^\hal$,  we easily deduce  that if $\int_{-\infty}^\infty T_2(v)\, dv \neq 1$ and $\int_{-\infty}^\infty \log |2\pi v|T_2(v)\, dv<\infty$, then
$$\frac{1}{N} \int_{D_0} \log \left|2\sin \frac{\pi (x-y)}{N}\right| T_2(x-y)\, dx\, dy + \log N \to \infty$$as $N \to \infty$, and we conclude as desired.
If $\int_{-\infty}^\infty T_2(v)\, dv =1$, then we may proceed and find
\begin{multline}\label{d0}
\frac{1}{N} \int_{D_0} \log \left|2\sin \frac{\pi (x-y)}{N}\right| T_2(x-y)\, dx\, dy + \log N=
 \int_{-\a_N}^{\a_N}  \log |2\pi v|T_2(v)\, dv \\
  + \log N \(1-\int_{-\a_N}^{\a_N} T_2(v)\, dv\)  +O\(\frac{\a_N}{N}  \log N \) + O\( \frac{\a_N^3}{N}\)   \end{multline}
 and we conclude (returning to \eqref{ew1d} and  \eqref{tocomp}, and using the assumptions)  that \eqref{resth1} holds.
\end{proof}

\subsection{Specific computations on the line}\label{sec:speci}
In this subsection, we use the result from the previous theorem to compute explicit asymptotic values for $\E \Wc$ for some well-known point processes, namely the homogeneous  Poisson process, and the $\beta$-sine processes with $\beta=1,2,4$. (Recall that $\Wc$ was defined in such a way that
for the lattice $\mz$ we have  $\Wc =  0$.)

The homogeneous Poisson process satisfies $\ro_1(x)=1$ and  $\ro_2(x,y)=1$, hence  $T_2=0$. We immediately deduce from \eqref{ew2d} that  $\E \Wc \to + \infty$. 
Hence, the Poisson process can be viewed as having the `value of $W$' equal to $+\infty$.

The $\beta$-sine processes for $\beta=1,2,4$ arise in random matrices as the local limits for random matrix ensembles with orthogonal, Hermitian, and symplectic symmetries, see \cite{mehta,agz,forrester} and references therein.
These are stationary processes whose correlations can be computed as follows: introduce the  kernels
\begin{align}\label{k2} &  K^{(2)}(x,y)= \frac{\sin \pi(x-y)}{\pi(x-y)} &\quad \text{for } \ \beta=2\\
\label{k1} & K^{(1)}(x,y)= \D\left(\begin{array}{lr}\D \frac{\sin\pi(x-y)}{\pi(x-y)}
&  \D\frac{\p }{\p x} \frac{\sin\pi(x-y)}{\pi(x-y)} \\[2mm]
\D\frac 1\pi\int_0^{\pi(x-y)} \frac{\sin t}{t}\, dt - \D\hal sgn(x-y) & 
\D\frac{\sin\pi(x-y)}{\pi(x-y)}\end{array}\right)   & \quad \text{for } \ \beta=1\\
\label{k4}& K^{(4)}(x,y)= \D\left(\begin{array}{lr}\D\frac{\sin 2\pi(x-y)}{2\pi(x-y)} & \D \frac{\p }{\p x} \frac{\sin 2\pi(x-y)}{2\pi(x-y)}\\[2mm]
\D\frac{1}{2\pi}\int_0^{2\pi(x-y)} \frac{\sin t}{t}\, dt &\D
\frac{\sin 2\pi(x-y)}{2\pi(x-y)}\end{array}\right)& \quad \text{for } \ \beta=4\end{align}
where all the indeterminacies $0/0$ at $x=y$ are resolved by continuity. \\
The $\beta=2$ sine process is a determinantal process with kernel $K^{(2)}(x,y)$, thus from \eqref{rodet}, its $2$-point correlation function is given by 
\begin{equation}\label{rodet2}
\ro_2(x_1,x_2)= \det \(K^{(2)}(x_i, x_j)\)_{i,j \in [1,2]}.
\end{equation}
The correlation functions for the $\beta=1,4$ sine processes have the form
\begin{equation}\label{qdet}
\ro_n(x_1,\dots, x_n)= \mathrm{qdet} \(K^{(\beta)}(x_i, x_j)\)_{i,j \in [1, n]},\qquad n=1,2,\dots
\end{equation}
where $\mathrm{qdet}$ denotes the quaternion determinant, see e.g. \cite[Section 6.1.1]{forrester}
for a definition. Alternatively, the right-hand side of \eqref{qdet} can be expressed
as the Pfaffian of a closely related matrix, cf. \cite[Proposition 6.1.5]{forrester}. 
Random point processes with correlation functions of such form are often called \emph{Pfaffian},
see \cite[Section 10]{Bor11} and references therein. 

The three processes above satisfy $\ro_1\equiv 1$, 
and their second cluster functions can be easily seen to be given by 
\begin{eqnarray}\label{t21}
& T_2^{(2)} (v)& =\(\frac{\sin \pi v}{\pi v}\)^2\qquad \text{for } \ \beta=2\\
\label{t22}
& T_2^{(1)}(v)& = \( \frac{\sin \pi v}{\pi v}\)^2 - \frac{1}{\pi} \frac{\p }{\p v} \frac{\sin \pi v}{\pi v} \(\int_0^{\pi v} \frac{\sin t}{t}\, dt- \frac{\pi}{2}sgn(v)\) \qquad \text{for } \ \beta=1\\
 \label{t24}& T_2^{(4)} (v) & =  \(\frac{\sin  2 \pi v}{2\pi v }\)^2- \frac{1}{2\pi} \frac{\p }{\p v} \frac{\sin 2\pi v}{2\pi v}\int_0^{2\pi v} \frac{\sin t}{t}\, dt \qquad \text{for } \ \beta=4
.\end{eqnarray}

\begin{pro} \label{pro23}
The $\beta$-sine processes for $\beta=1,2,4$ satisfy the assumptions of Theorem \ref{th1}, and
\begin{eqnarray*}
&\D  \lim_{N\to \infty} \E \Wc  & = 1-\gamma
 \qquad \text{for } \ \beta=2\\
 & \D  \lim_{N\to \infty} \E \Wc  & = 2-\gamma - \log 2
\qquad \text{for } \ \beta=1\\
 &   \D\lim_{N\to \infty} \E \Wc & = \frac{3}{2}-\log 2-\gamma\qquad \text{for } \ \beta=4,\end{eqnarray*}
where $\gamma$ is the Euler constant.
\end{pro}

Before stating the proof, we recall some integrals of classical functions that we will need. We state them without a proof and refer to \cite{pbm1,pbm2}.
\begin{eqnarray}
\label{dirichlet}
   \int_0^{+\infty} \frac{\sin v }{v} \, dv & = & \frac{\pi}{2}
   \\
  \label{2632}
  \int_0^{+\infty} \frac{\sin v}{v} \log v \, dv   &  =  &\frac{\pi}{2}\gamma
  \\
 \label{2569} \int_0^{+\infty}\(\frac{\sin v}{v}\)^2 \, dv & = & \frac{\pi}{2}
 \\
\label{225} \int_0^{+\infty} \(\frac{\sin v}{v}\)^2 \log v \, dv & = & - \frac{\pi}{2} (\gamma+\log 2 - 1).\end{eqnarray}
 These formulas can be found in \cite{pbm1}, they are respectively (2.5.3.12), (2.6.32.3), (2.5.6.9), and (2.6.32.7).
 Finally we need a few more  integrals that are based on
 $$Si(x)=\int_0^x \frac{\sin t}{t}\, dt , \quad si(x)= Si(x)- \frac{\pi}{2} = -\int_{x}^{+\infty}  \frac{\sin t}{t} \, dt$$ and
 $$Ei(x)= \int_{-\infty}^x \frac{e^t}{t}\, dt.$$
 These are
 \begin{eqnarray}
 \label{26411}
 \int_0^{+\infty} \frac{ si(v) \sin v}{v^2+z^2} \, dv & = &\frac{\pi}{2z}\sinh(z)Ei(-z)\\
\label{26210}
  \int_0^{+\infty} \frac{v \, si(v)}{v^2+z^2} \, dv & = & \frac{\pi}{2}\,Ei(-z)\\
 \label{264}  \int_0^{+\infty}\frac{Si(v) \sin v  }{v^2+z^2 } \, dv &=& \frac{\pi}{4}e^{-z}\,\frac{Ei(z)-Ei(-z)}{z}\,,\end{eqnarray}
  cf. respectively (2.6.4.11), (2.6.2.10) and (2.6.4.16) in \cite{pbm2}.

\begin{proof}[Proof Proposition \ref{pro23}]
We start with the simplest case.
\\
{\it Case $\beta=2$:} First it is easy to check that assumptions \eqref{basicass} as well as assumption 1) of Theorem \ref{th1} are verified. Also, we have  $T_2(v)=O(|v|^{-2})$ as $v\to \infty$, 
and $1-\int_{-\a_N}^{\a_N} T_2=O((\a_N)^{-1})=o((\log N)^{-1})$, thus assumption 2) is  implied by \eqref{rem2.1}.
According to Theorem \ref{th1} we then have
$$  \lim_{N\to \infty} \E\Wc   = \int_{-\infty}^\infty \log |2\pi v|
\(\frac{\sin \pi v}{\pi v}\)^2
 \, dv. 
$$
{}From \eqref{225}, we obtain
 \begin{equation}\label{intsinlog}
 \int_{-\infty}^\infty \left(\frac{\sin\pi v}{\pi v}\right)^2 \log|2\pi v|\, dv=1-\gamma,\end{equation}
 hence the result.
 \\
 {\it Case  $\beta=1$:}
 Similarly to the case of $\beta=2$, one checks that $T_2(v)=O(|v|^{-2})$. Indeed, note that $\int_0^{\pi |v|} \frac{\sin t}{t} \, dt - \frac{\pi}{2} = \int_{\pi |v|}^{+\infty} \frac{\sin t}{t}\, dt= \frac{\cos (\pi |v|) }{\pi |v|} -\int_{\pi |v|}^{+\infty} \frac{\cos t}{t^2}= O({|v|^{-1}})$.
 Assumptions \eqref{basicass} and  1), 2) are then  verified as in the $\beta =2$ case.
  Next we check that $\int T_2^{(1)}=1$.
 By evenness and integration by parts
 \begin{multline*}\int_{-\infty}^\infty T_2^{(1)}(v)\, dv= 2 \int_{0}^\infty
 \left( \( \frac{\sin \pi v}{\pi v}\)^2 - \frac{1}{\pi} \frac{\p }{\p v} \frac{\sin \pi v}{\pi v} \(\int_0^{\pi v} \frac{\sin t}{t}\, dt- \frac{\pi}{2}\)\right)dv\\
 =4\int_0^\infty \( \frac{\sin \pi v}{\pi v}\)^2 \, dv -1 =1.\end{multline*}
  The convergence is also fast enough, since $T_2(v) = O(|v|^{-2})$.

 According to Theorem \ref{th1} we thus have
 \begin{multline*} \lim_{N\to \infty} \E \Wc   =  \int_{-\infty}^\infty \log | 2\pi v| \( \( \frac{\sin \pi v}{\pi v}\)^2 - \frac{1}{\pi} \frac{\p }{\p v}
 \frac{\sin \pi v}{\pi v} \(\int_0^{\pi v} \frac{\sin t}{t}\, dt- \frac{\pi}{2}sgn(v)\)\) \, dv\\ =
2  \int_{0}^\infty \log |2\pi  v| \( \( \frac{\sin \pi v}{\pi v}\)^2 - \frac{1}{\pi} \frac{\p }{\p v} \(\frac{\sin \pi v}{\pi v}-1\) \(\int_0^{\pi v} \frac{\sin t}{t}\, dt- \frac{\pi}{2}\)\) \, dv.
 \end{multline*}
 Integrating by parts, we are led to
 \begin{multline}\label{rfg1}
 \lim_{N\to \infty} \E \Wc= 4  \int_{0}^\infty \log (2 \pi v)\(\frac{\sin \pi v}{\pi v}\)^2\, dv+\frac{2}{\pi}   \int_0^\infty \(\frac{\sin \pi v}{\pi v}-1\) \( \int_0^{\pi v} \frac{\sin t}{t}-\frac{\pi}{2}\) \frac{dv }{v}\\
-\frac 2\pi\int_0^{+\infty} \log(2\pi v)\,\frac{\sin \pi v}{v}\,dv .\end{multline}
 In view of \eqref{intsinlog} the first term on the right-hand side is equal to $2-2\gamma$.
 The second is equal to
 \begin{multline*} \frac{2}{\pi} \int_0^{+\infty} \(\frac{\sin \pi v}{\pi v}-1\) si(\pi v) \frac{dv}{v} = \frac{2}{\pi} \int_0^{+\infty} \( \frac{\sin u}{u}-1\) si(u) \frac{du}{u}
\\  = \frac{2}{\pi}  \lim_{z\to 0} \int_0^\infty \( \frac{si(u) \sin u }{u^2+z^2} - \frac{si(u) u}{u^2+z^2}\) \, du.
\end{multline*}
With \eqref{26411} and \eqref{26210},
$$ \frac{2}{\pi} \int_0^{+\infty} \(\frac{\sin \pi v}{\pi v}-1\) si(\pi v) \frac{dv}{v} = \lim_{z\to 0}  Ei(-z) \( \frac{\sinh(z)}{z}-1\).$$
On the other hand, near $x=0$ one has \begin{equation}\label{expei}
Ei(x)=\gamma + \log |x|+ \sum_{n=1}^\infty \frac{x^n}{n n!}, \end{equation} hence the above right-hand side limit is $0$. The second term in \eqref{rfg1} thus vanishes.  From \eqref{2632} and \eqref{dirichlet}, the third one is equal to   $\gamma-\log 2.$ This concludes the case $\beta=1$.
 \\
 {\it Case $\beta=4$:}
 Assumptions \eqref{basicass}  and 1) are easily verified, we proceed to 2).
 We may write
 $$T_2(v)= \( \(\frac{\sin 2\pi v}{2\pi v}\)^2 - \frac{1}{2\pi} \frac{\p }{\p v} \frac{\sin 2\pi v}{2\pi v} \( \int_0^{2\pi v} \frac{\sin t}{t}\, dt - \frac{\pi}{2} sgn(v) \) \) - \frac{1}{4}\frac{\p}{\p v} \frac{\sin 2\pi v}{2\pi v}sgn(v)
.
$$
The first part is $O(|v|^{-2})$ just as in the case above, so it remains to check that
$$\int_{\a_N}^A \frac{\p }{\p v}\frac{\sin 2\pi v}{2\pi v} \log \left|2\sin \frac{\pi v}{N}\right|\, dv=o(1)$$ uniformly in $A\in [\a_N, N-\a_N].$
Integrating by parts, we have
\begin{multline*}
\int_{\a_N}^A \frac{\p }{\p v}\frac{\sin 2\pi v}{2\pi v} \log \left|2\sin \frac{\pi v}{N}\right|\, dv\\
= \frac{\sin 2\pi A}{2\pi A} \log \left|2\sin \frac{\pi A}{N}\right|- \frac{\sin 2\pi \a_N}{2\pi \a_N}\log \left|2\sin \frac{\pi \a_N}{N}\right| - \int_{\a_N}^A \frac{\sin 2\pi v}{2 v} \frac{\cos \frac{\pi v}{N}}{N\sin \frac{\pi v}{N}}\, dv.\end{multline*}
In $[\a_N, A]$ we may bound from below $\sin \frac{\pi v}{N} $ by $\sin \frac{\pi \a_N}{N}$, which is asymptotically
 equivalent to $\frac{\pi \a_N}{N}$ as $N \to \infty$. Hence the integral on the right hand side may by bounded by
$\frac{C}{\a_N}\int_{\a_N}^N\frac{dv}{v}\le \frac{C\log N}{\a_N} =o(1)$ in view of assumption 4).
The other terms are also easily found to be $o(1)$ by a similar argument.

We next check that $\int T_2=1$ with fast enough convergence.
 By evenness and integration by parts
 \begin{multline*}\int_{-\a_N}^{\a_N} T_2^{(4)}(v)\, dv= 2 \int_{0}^{\a_N}\left(
  \( \frac{\sin 2\pi v}{2\pi v}\)^2 - \frac{1}{2\pi} \frac{\p }{\p v} \frac{\sin 2 \pi v}{2 \pi v} \(\int_0^{2\pi v} \frac{\sin t}{t}\, dt\)\right) dv\\
 =4\int_0^{\a_N} \( \frac{\sin 2\pi v}{2\pi v}\)^2 \, dv - \frac{1}{2\pi} \frac{\sin 2\pi \a_N}{\a_N} \int_0^{2\pi \a_N} \frac{\sin t}{t}\, dt  =1+O((\a_N)^{-1}).\end{multline*}
 According to Theorem \ref{th1} we thus have
 \begin{equation*} \lim_{N\to \infty} \E \Wc =  \int_{-\infty}^\infty \log |2\pi v|
 \(
  \(\frac{\sin  2 \pi v}{2\pi v }\)^2- \frac{1}{2\pi} \frac{\p }{\p v} \frac{\sin 2\pi v}{2\pi v}\int_0^{2\pi v} \frac{\sin t}{t}\, dt\)
 \, dv.\end{equation*}
 Using evenness and integration by parts as above, we find
\begin{equation*}
 \lim_{N\to \infty} \E \Wc
 =4 \int_0^\infty  \(\frac{\sin  2 \pi v}{2\pi v }\)^2 \log | 2\pi v|\, dv
 +\frac{1}{\pi} \int_0^\infty \frac{\sin 2\pi v}{2\pi v^2}\( \int_0^{2\pi v} \frac{\sin t}{t}\, dt\) \, dv.\end{equation*}

By \eqref{225} the first term on the right-hand side is equal to $-\gamma - \log 2 +1$. By change of variables,  the second term is equal to
$$\frac{1}{\pi} \int_0^{+\infty} \frac{Si(u) \sin u}{u^2}\, du= \frac{1}{\pi} \lim_{z\to 0} \int_0^{+\infty}   \frac{Si(u) \sin u}{ u^2+z^2}\, du= \frac{1}{4} \lim_{z\to 0} e^{-z} \frac{Ei(z)-Ei(-z)}{z}  $$
by \eqref{264}.  Combining with \eqref{expei} we find that the second term is equal to $\hal$ and we conclude the proof.
 \end{proof}
 As expected $\lim_{N\to \infty}\E \Wc$ decreases as $\beta=1,2,4$ increases.

\subsection{Expectation for two-dimensional processes : theoretical formula}
In the plane, the computations are easier because we can take advantage of the fast (exponential) decay of the correlation kernels.
\begin{theo}\label{th2}
 Consider a random point process  $\X$ in   the plane, with a  one point correlation function $\ro_1(x) $ and a two-point correlation function $\ro_2(x,y)$, satisfying \eqref{basicass}.  We identify the plane with the complex plane $\mc$.  Under the assumption
   $$\int_{\mr^2} |v|^k |T_2(v)|\, dv <\infty \qquad \text{for } \ k=1,2,3;$$
 the following holds:
 \\
 - if $\D \int_{\mr^2} T_2(v)\, dv=c \neq 1 $, we have $\E\Wc \to  \infty $ as $N\to \infty$;\\
 - if $\D \int_{\mr^2} T_2(v)\, dv =1$ and $1-\int_{[-N,N]^2} T_2(v)\, dv=  o((\log N)^{-1})$, then  $\lim_{N\to \infty} \E \Wc $ exists and is finite if and only if   $\int_{\mr^2} T_2(v)\log |v| \, dv$ converges, and if so,  then
 \begin{equation}\label{resth2}
 \lim_{N\to \infty} \E \Wc =  \int_{\mr^2} \log | v|\, T_2(v) \, dv.\end{equation}
\end{theo}
\begin{proof}
Returning to \eqref{ew2d} we have to compute
$$\lim_{N\to \infty} -\frac{1}{2\pi N^2}\int_{[0,N]^2 \times [0,N]^2 } E(x-y) T_2(x-y)\, dx \, dy +
\log \frac{N}{2\pi \eta(i)^2}.$$

Making the change of variables $(u,v)= (x+y,x-y)$, we have
$$\int_{[0,N]^2 \times [0,N]^2}E_N(x-y)   T_2(x-y)\, dx \, dy = \frac{1}{4} \int_{v\in [-N,N]^2}\int_{ u \in S_N(v)}E(v)  T_2(v)\,du \, dv,$$
where $S_N(v)=\{ x+y: x\in [0,N],\, y\in [0,N]^2,\, x-y=v\}$. We may compute that
$|S_N(v)|= 4N^2 - 2N|v_1|- 2N|v_2|+|v_1||v_2|,$
so
\begin{multline}\label{et1}
\int_{[0,N]^2\times [0,N]^2} E_N(x-y)
 T_2(x-y)\, dx \, dy \\ = \frac{1}{4}\int_{[-N,N]^2} \left(4N^2 - 2N|v_1|- 2N|v_2|+|v_1||v_2|\right) E_N(v) T_2(v)\, dv.\end{multline}

Next we return to \eqref{klf} where $f$ is given by \eqref{f} and perform an asymptotic analysis as $N\to \infty$.
We have
$$p^{1/2}-p^{-1/2}= e^{i\pi \frac{\overline{x}}{N} } - e^{-i \pi \frac{\overline{x}}{N}}=
2i \pi \frac{\overline{x}}{N}+O\(\frac{|x|^2}{N^2}\),$$
while, since $p=1+ O(\frac{|x|}{N})$ we may write (with $q=e^{-2\pi}$) 
$$
q^{1/12}\prod_{k\ge 1} (1-q^k p)(1-q^k/p)= \( q^{1/12} \prod_{k \ge 1} (1-q^k) \)  (1+ O(p-1))= \eta(i)^2 + O\left(\frac{|x|}{N}\right),$$ hence
$$f\left(\frac{\overline{x}}{N}, i\right) = 2i \pi \frac{\overline{x}}{N} \eta(i)^2 \left(1+O\left(\frac{|x|}{N}\right)\right).$$
Inserting into \eqref{klf} and combining with $e^{-\pi (\Im \frac{x}{N})^2 }=1+O(\frac{|x|^2}{N^2})$ we obtain
\begin{equation}\label{asE}
E_N(x)= -2\pi  \log |x|-2\pi \log \frac{2\pi \eta(i)^2 }{N}+O\left( \frac{|x|}{N}\right)\quad \text{as} \ N \to \infty.\end{equation}

Inserting this into \eqref{et1}, we are led to
\begin{multline*}
\int_{[0,N]^2\times [0,N]^2 } E_N(x-y) T_2(x-y)\, dx \, dy \\
=  \int_{[-N,N]^2 }   \left(N^2 - \hal N|v_1|- \hal N|v_2|+\frac{1}{4}|v_1||v_2|\right) \(  -2\pi  \log |v|-2\pi \log \frac{2\pi \eta(i)^2 }{N}+O\left( \frac{|v|}{N}\right)\) T_2(v) \, dv.\end{multline*}
Therefore
\begin{multline*} - \frac{1}{2\pi N^2} \int_{[0,N]^2 \times [0,N]^2} E_N(x-y) T_2(x-y) \, dx\, dy +  \log \frac{N}{2\pi \eta(i)^2} \\=  \log \frac{N}{2\pi \eta(i)^2}  \( 1-\int_{[-N,N]^2 }\(1 +O\left(\frac{|v|}{N}\right) +O\left(\frac{|v|^2}{N^2}\right)\) T_2(v)\,dv \)\\
+ \int_{[-N,N]^2} \(  \log |v| + O\left(\frac{|v|}{N}\right)  \)\(1 +O\left(\frac{|v|}{N}\right) +O\left(\frac{|v|^2}{N^2}\right) \) T_2(v) \,     dv.\end{multline*}
Using the assumption we have  \begin{eqnarray*}
& \int_{[-N,N]^2} \frac{|v|}{N}T_2(v)\, dv=o(1), \\
& \int_{[-N,N]^2 } \frac{|v|^2}{N^2} T_2(v)\, dv=o(1),\\ 
&\int_{[-N,N]^2 } \frac{|v|^3}{N^3}T_2(v)\, dv=o(1),\\  
&\int_{[-N,N]^2 } \frac{|v|}{N} \log |v|T_2(v)\, dv=o(1),  \\
&\int_{[-N,N]^2 } \frac{|v|^2}{N^2} \log |v|T_2(v)\, dv=o(1).\end{eqnarray*}
 It follows that, as $N \to \infty$,
\begin{multline*} - \frac{1}{2\pi N^2} \int_{[0,N]^2\times [0,N]^2 } E_N(x-y) T_2(x-y) \, dx\, dy + \log \frac{N}{2\pi \eta(i)^2} \\=  \log \frac{N}{2\pi \eta(i)^2}  \( 1-\int_{[-N,N]^2 } T_2(v)\,dv \)
+ \int_{[-N,N]^2}  \log |v|  T_2(v) \,     dv +o(1).\end{multline*}
The result then easily follows.\end{proof}

\subsection{Specific computations  in the plane}
We turn to computing that limit for two specific processes. The first one is 
the determinantal random point process with correlation kernel
\begin{equation}\label{kginibre}
K(x,y)= e^{-\frac{\pi }{2} (|x|^2 + |y|^2 - 2x\overline{y})}.\end{equation}
This process arises in random matrices as the local  limit of the complex  Ginibre ensemble, see e.g. \cite[Proposition 15.2.3]{forrester}, and thus it is sometimes called the Ginibre point process.
From the determinantal structure of the correlation functions, cf. \eqref{rodet}, 
we have $\ro_2(x,y)=1-|K(x,y)|^2=
1- e^{-\pi |x-y|^2}$ and $T_2(v)=e^{-\pi |v|^2} $.
This easily satisfies all the assumptions of Theorem \ref{th2} (in particular $\int T_2=1$) and we obtain
\begin{pro}
The determinantal process with kernel \eqref{kginibre} satisfies
$$\lim_{N\to \infty} \E \Wc= -\frac{1}{2}\left(\gamma+\log \pi\right).$$

\end{pro}

This statement can be compared to a computation done in \cite{jancovici}, see also 
\cite[Ex.15.3.1(iv)]{forrester}. 

\begin{proof} According to Theorem \ref{th2} it suffices to compute
$$\int_{\mr^2}\log |v| e^{-\pi |v|^2}\, dv= \int_0^\infty \log r \, e^{-\pi r^2 } 2\pi r \, dr=\int_0^\infty \hal (\log s - \log \pi)e^{-s}ds ,$$ using the change of variables $s=\pi r^2$.
We have $\int_0^\infty e^{-s} \log s\, ds=-\gamma$, and the result follows.\end{proof}

The second one is the process of the zeros of a Gaussian analytic function (often denoted GAF).  
It consists of the random zeros of the analytic function  
$\sum_{n=0}^\infty \frac{\xi_n}{\sqrt{n!}} z^n$ when the $\xi_n$ are i.i.d Gaussians suitably normalized, and 
it is a stationary process in the plane.    The general background can be found e.g. in  \cite{Hou06}.

The second cluster function for the  process, when the density $\ro_1$ is taken to be $1$,  is given according to \cite{forresterh} by
$$T_2(x)=1-h\(\frac{\pi|x|^2}{2}\)$$
where $$h(x)= 1+ \hal \frac{d^2}{dx^2} \( x^2 (\coth x -1 )\)  .$$
It is easy to check that the assumptions of Theorem \ref{th2} are satisfied, and we deduce 
\begin{pro}
The  ``zeros of Gaussian analytic functions" process satisfies 
$$\lim_{N\to \infty} \E \Wc= - \frac{1}{2}(1+\log \pi)
.$$
\end{pro}
\begin{proof}
To check that we may apply Theorem \ref{th2} we first compute 
\begin{eqnarray*}
\int_{\mr^2} T_2(v)\, dv & = & \int_{\mr^2} \(1- h\(\frac{\pi |v|^2}{2}\)\) \, dv\\
& = & \int_0^\infty \( 1- h\(\frac{\pi r^2}{2}\) \) 2\pi r \, dr \\
& = & 2 \int_0^\infty    \( 1-h(u)\) \, du\\
& = & -  \left[ \frac{d}{dx} ( x^2 (\coth x -1))\right]_0^\infty \\
& = &  - \left[ 2x(\coth x-1) + x^2 (1-\coth^2 x)\right]_0^\infty \\
 & = & 1
 \end{eqnarray*}
 where we have used the  change of variables $u = \pi r^2/2$ and the asymptotic relation 
$ \coth x\sim \frac{1}{x}$ as $x \to 0$.
 It is also easy to check that the convergence of $\int T_2$ is exponential, hence fast enough, and we may apply Theorem \ref{th2}.
 
This yields
\begin{eqnarray*}
\lim_{N\to \infty} \E \Wc& =&  \int_{\mr^2} \log |v|T_2(v)\, dv\\
& = &  2\pi \int_0^\infty \log r\(1-h\(\frac{\pi r^2}{2}\)\) r\, dr \\
& = & 2 \int_0^\infty \log \sqrt{\frac{2u}{\pi}}\,  (1- h(u))\, du\\
& = & - \frac{1 }{2} \int_0^\infty\( \log\frac{ 2u}{\pi}  \)\frac{d^2}{du^2} (u^2 (\coth u -1)) \, du.\end{eqnarray*}

Let us now compute 
\begin{multline*}
\int_\ep^\infty \(\log \frac{2u}{\pi}\) \frac{d^2}{du^2} (u^2 (\coth u -1)) \, du\\
=  \left[  \(\log \frac{2x}{\pi}\) \frac{d}{dx} ( x^2 (\coth x -1) ) \right]_{\ep}^\infty - \int_{\ep}^\infty \frac{1}{u}  \frac{d}{du} ( u^2 (\coth u -1))\, du \\
=  - \log \frac{2\ep}{\pi}  \( 2\ep (\coth \ep -1) + \ep^2 (1- \coth^2 \ep)\)    - \left[ x ( \coth x -1) \right]_{\ep}^\infty  - \int_{\ep}^\infty (\coth u - 1) \, du\\
 = - \log \frac{2\ep}{\pi}  (1+O(\ep)) + 1+O(\ep) - \left[\log \sinh x - x \right]_{\ep}^\infty.\end{multline*}
 Taking the limit $\ep \to 0$, we conclude 
 $$\lim_{N \to \infty} \E \Wc = - \frac{1}{2}(1+\log \pi) .$$
 \end{proof} 
 It is well known (cf. \cite{ns} and references therein) that 
the  ``GAF  process" is more ``rigid"  than the ``Ginibre point process"  (cf. also recent work \cite{gnps}). We just demonstrated here  it has more order  via the renormalized energy.

\section{Optimization over determinantal processes}\label{optim}
As explained in the introduction and Section \ref{sec:background}, the question of minimizing $W$ is an important one, and open in dimension 2. It thus seems interesting to try to minimize $\lim_{N\to \infty} \E \Wc$ as expressed in \eqref{resth1} and \eqref{resth2}, over a subclass of processes.

In this section we show that we can characterize the minimizer of this expression over  the class of 
determinantal random point processes   whose correlation kernel $K(x,y)$ is Hermitian and 
translation invariant, i.e. 
$K(x,y)= k(x-y)$ for some function $k$.  For those processes, we have $T_2=k^2$. Note, however, that the important 
determinantal process with kernel \eqref{kginibre} is not in this class: while all its
correlation functions are translation invariant, the correlation kernel is not. 

We prove the following statement. The proof relies on a rearrangement inequality.
\begin{theo}\label{th3}
Let $\cal{K}$ be the class of determinantal processes  on the real line, respectively the plane, 
with self-adjoint  translation-invariant kernels $K(x,y)=k(x-y)$, and $k(v)\in L^2(\mathbb{R}^d)$, $d=1$ or $d=2$,  such that
\begin{enumerate}
\item[1)] $\rho_1(x)=k(0)=1$ and $\int_{\mathbb{R}^d} k^2(x)\, dx=1$;
\item[2)] $ \int_{\mathbb{R}^d} \log |x|k^2(x)\, dx<\infty $.
\end{enumerate}
Let ${\cal F}(k)=
\int_{\mr^d} \log |x|k^2(x)\, dx $. Then for
any process from $\cal K$ with correlation kernel $K(x,y)=k(x-y)$, we have
$${\mathcal F} (k)\ge {\cal F} ( \widehat{\indic_{B}}),$$
 where $B$ is the ball centered at $0$ of volume one, and \ $\widehat{\cdot} $ is the  Fourier transform.
Thus, on the real line ${\cal F}(k)$ is minimized over $\cal K$ by the $\beta=2$ sine process, while
on the plane the minimizing determinantal process has the kernel given by 
$k(v)=\widehat{\indic_B}(v)= \frac{J_1(2\sqrt{\pi} |v| ) }{\sqrt{\pi}|v|}$, where $J_1$ is a Bessel function of the first kind.
\end{theo}
\begin{remark} Condition 1) says that our processes have density 1, which we have assumed throughout,
and that $\lim_{N\to\infty} \E\Wc$ is finite, cf. Theorems \ref{th1} and \ref{th2}. The functional $\cal F$ coincides with
$\lim_{N\to\infty} \E \Wc$, given that the decay assumptions are satisfied. 
\end{remark}
\begin{remark} Numerical integration shows that in dimension 2, ${\cal F} ( \widehat{\indic_{B}})\approx
      -0.65$, which is greater than $\lim_{N\to\infty} \E\Wc$ for the Ginibre ensemble 
$(=-\tfrac 12(\gamma+\log \pi)\approx -0.86)$ and for the zeroes of the Gaussian analytic function
$(=-\tfrac 12(1+\log\pi)\approx -1.07)$. Thus, the two latter processes are ``more rigid''.
     \end{remark}

\begin{proof}[Proof of Theorem \ref{th3}] 
Let us denote by $f$ the inverse Fourier transform of $k$, and by $T_K$ the integral operator 
corresponding to $K$ via $T_K(\varphi)= \int K(x,y)\varphi(y)\, dy.$ We have
$$T_K(\varphi)=  \int k(x-y)\varphi(y)\, dy = k * \varphi.$$
By Macchi-Soshnikov's theorem, see \cite{Sos00}, any self-adjoint translation-invariant
correlation kernel of a determinantal process gives rise to an integral operator
with spectrum between $0$ and $1$. Hence, the spectrum of $T_K$ is in $[0,1]$. Since $T_K$ is also the convolution by $k$, this implies that $f=\check{k}$ takes values in $[0,1]$. Moreover $f \in L^2 $ with $\int |f|^2=\int k^2=1$, and $\int f=k(0)=1$  from assumption  1). 
A function $f $ with values in $[0,1]$ which satisfies $\int f=\int f^2=1$ can only be a characteristic function of a set $A$, denoted $f=\indic_A$, with $A$ measurable of measure 1. Writing $k_A$ for the corresponding $k$, we may write
$$k_A(x)= \int e^{- 2i \pi \xi \cdot x} \indic_{A}(\xi)\, d\xi.$$
There remains to optimize over $A$, measurable set  of measure 1,
\begin{equation}\label{ia}
 I(A):= \int_{\mr^d} k_A^2(x) \log |x|\, dx.
\end{equation}
Let  $A$ be measurable of measure 1, such that $I(A) < +\infty$.
Noting that $k_A $ is an $L^\infty$ (and also continuous) function, for every $\a>0$ the integrals
$$\int_{\mr^d} k_A^2(x) \frac{1-|x|^{-\alpha}}{\a} \, dx$$
then also converge, by comparison.
Given any $\tau\in \mr$, $(e^{\tau h}-1)/h$ converges to $\tau $ monotonically as $h \to 0$  (it suffices to check that
 this function  is increasing in $h$). It then follows that in each of  the domains $|x|<1$
and $|x|\ge 1$, we have
$\frac{1-|x|^{-\alpha}}{\a} \to \log |x|$ as $\a \to 0$, monotonically. Splitting the integrals as sums over these two regions,
it follows by monotone convergence theorem that
\begin{equation}
\label{cvmonotone}
I(A)= \lim_{\a\to 0}  \int_{\mr^d}  k_A^2(x) \frac{1-|x|^{-\alpha}}{\a} \, dx.\end{equation}
We then remark  (in view of the formula $\widehat{f*g}=\hat{f}\hat{g}$)  that
$k_A^2(x)$ is the Fourier transform of
$$f_A(x)=\int_{\mr^d}\indic_A(y)\indic_A (y-x)\, dy= |(A+x)\cap A|.$$
We next claim that $f_A$ is a continuous function.
 First, consider the case where $A$ is an open set.    Then, as $x\to x_0$ we have $\indic_A(y) \indic_A(y-x)
\to \indic_A (y) \indic_A(y-x_0)$ almost everywhere, by openness of $A$, while
$|\indic_A (y)\indic_A(y-x)|\le \indic_A(y) $ and $\indic_A \in L^1$. The claim is thus true by dominated
convergence theorem.
Second, if $A$ is a general measurable set, by outer regularity of the measure we may approximate
it by an open set $U$ such that $A\subset U$ and $|U\backslash A|<\ep$.
Then it is immediate that for any $x$,  $|f_A(x)-f_U(x)|= ||(A+x)\cap A|- |(U+x)\cap U||\le 2\ep$.
 Since $f_U$ is continuous and $\ep $ is arbitrary, it  follows that $f_A$
 has to be continuous too as a uniform limit of continuous functions. The claim is proved.
We also note that $f_A(0)=|A|=1$.

The next ingredient is that  in dimension 1
\begin{equation}\label{f1d}
\widehat{(1-|x|^{-\alpha})}= \delta- 2  \Gamma(-\alpha+1) \sin \frac{\pi \a}{2}  (2\pi |\xi|)^{\a-1}, \end{equation}
while in dimension 2
\begin{equation}\label{f2d}
 \widehat{(1-|x|^{-\alpha})}= \delta - \pi^{\a-1}\frac{\Gamma(\frac{2-\a}{2})}{\Gamma(\frac{\a}{2})} |\xi|^{\a-2}
\end{equation}
For a reference see e.g. \cite[Chapter 5]{edwards}, or \cite[page 113]{sch}.

Let us continue with the one-dimensional case.
We deduce from the above facts that
\begin{equation}\label{fourier3}
I_\alpha(A):=  \int_{\mr} k_A^2(x) \frac{1-|x|^{-\alpha}}{\a} \, dx= \frac{1}{\a}\int_{\mr}  f_A(\xi)\(
 \delta- 2 \Gamma(-\alpha+1)  \sin \frac{\pi \a}{2}  (2\pi |\xi|)^{\a-1} \)\, d\xi.\end{equation}
The relation can be justified
by convoluting $ \delta- 2 \Gamma(-\alpha+1) \sin \frac{\pi \a}{2}  (2\pi |\xi|)^{\a-1}$ with a Gaussian
kernel approximating $\delta$ at scale $\ep$,
using the fact that $\int \hat{f} g=\int f \hat{g}$ in the Schwartz class, the continuity of $f_A$
and then letting $\ep \to 0$ on both sides. Moreover, this argument shows that
$\int f_A(\xi)  |\xi|^{\a-1}\, d\xi $ is convergent.

Using $f_A(0)=1$ and Fubini's theorem, we may rewrite \eqref{fourier3} as
\begin{multline}\label{fourier4}
I_\alpha(A) = \frac{1}{\a}\( 1- 2 \Gamma(-\alpha+1) \sin \frac{\pi \a}{2}
\int_{\mr^2}  \indic_A(y)\indic_A(y-x) (2\pi |x|)^{\a-1}  \, dx\, dy\) \\
= \frac{1}{\a}- \frac{2}{\a} \Gamma(-\a +1) \sin \frac{\pi \a}{2}
\int_{\mr^2} \indic_A(y)\indic_A(z) (2\pi |y-z|)^{\a-1}\, dz\, dy.
\end{multline} Notice that $\frac{1}{\a}  \Gamma(-\alpha+1) \sin  \frac{\pi \a}{2} \sim \frac{\pi}{2}$ as $\a\to 0$
so for $\a$ small enough, $- \frac{2}{\a} \Gamma(-\alpha+1) \sin \frac{\pi \a}{2}(2\pi |y-z|)^{\a-1}$ is increasing in
$|y-z|$.
Now Riesz's rearrangement inequality (see \cite[Theorem 3.7]{ll})
 asserts that a quantity of the form of the right-hand side of \eqref{fourier4}
 is always decreased by changing $\indic_A$ into its symmetric rearrangement
 $(\indic_A)^*=\indic_{A^*}$.
This means that for all $\alpha$ small enough,
$I_\alpha(A) \ge I_\alpha(A^*)$. But  $I(A)= \lim_{\a\to 0} I_\alpha(A)$ hence the same is true also for $I$
i.e. $I(A) \ge I(A^*).$  The symmetric rearrangement $A^*$ of $A$ is the ball centered at $0$ and of volume $|A|=1$.
We  have thus
 found that $\int k_A^2(x) \log |x|\, dx$ is minimal when $A$
is the ball centered at $0$ and of volume 1.
In dimension one, the Fourier transform of $\indic_{[-\hal, \hal]}$ is $\frac{\sin \pi x}{\pi x}$, which corresponds to the determinantal process with $K(x,y)=\frac{ \sin \pi (x-y)}{\pi(x-y)}$, that is the sine process  (for $\beta=2$).

In dimension 2, the argument is exactly parallel, starting again from \eqref{f2d}. \end{proof}
\section{Computations of variance of $\Wc$}\label{sec4}
In Section \ref{sec3} we dealt with the expectation of $W$. 
In this section we turn to examining its variance in the sense of computing 
$\lim_{N\to \infty} \V(\Wc) $ for the same specific random
point processes.

In what follows we will need the formalism of (higher) cluster functions to efficiently
deal with the $k$-point correlation functions for $k=2,3,4$; we refer to \cite{tw,faris}
for details and further references on this formalism. 

For any nonempty subset $S=\{i_1, \dots, i_k\}$ of  $\{1, \dots, N\}$ we write
$\ro_S=\ro_k(x_{i_1}, \dots, x_{i_k})$, where $\ro_k$ is the $k$-point correlation function,
 and define the $n$-point cluster function as
\begin{equation}\label{deftn}
T_n(x_1, \dots, x_n)=\sum (-1)^{n-m} (m-1)! \ro_{S_1} \dots \ro_{S_m}\end{equation}
with the sum running over all partitions of  $\{1, \dots, n\}$ into nonempty subsets $S_1, \dots , S_m$.
From the $T_n$, the $\ro_n$ can be recovered through the reciprocal formula
\begin{equation}\label{reciprocal}
 \ro_n=\sum (-1)^{n-m}T_{S_1}\dots T_{S_m}.
\end{equation}
If a random point process is determinantal (cf. \eqref{rodet}) with correlation kernel $K$, 
then (see e.g. \cite{tw}) for any $k\ge 1$
\begin{equation}\label{tdet}
T_k(x_1, \dots, x_k)=\frac{1}{k}\,\sum_{\sigma\in \mathbf{S}_k}K(x_{\sigma(1)}, x_{\sigma(2)})\dots K( x_{\sigma(k)}, x_{\sigma(1)}),\end{equation}
where $\mathbf{S}_k$ denotes the symmetric group on $k$ symbols.

If the correlation functions of a point process are given by quaternion determinants
(cf. \eqref{qdet}) then (see e.g. \cite{tw}) for any $k\ge 1$
\begin{equation}\label{tqdet}
T_k(x_1, \dots, x_k)=\frac{1}{2k}\,Tr\sum_{\sigma\in \mathbf{S}_k}K(x_{\sigma(1)}, x_{\sigma(2)}) \dots K( x_{\sigma(k)}, x_{\sigma(1)}).\end{equation}

\begin{lem}\label{prform}
We have
$$\V\(\sum_{i\neq j}G_N(a_i-a_j)\)= I_1 + \dots + I_6$$ where
 \begin{align}
 \label{t1} &
I_1=  2 \int_{[0,N]^2} G_N(x-y)^2 \, dx\, dy\\
  \label{t2}  & I_2= - 4 \int_{[0,N]^3} G_N(x-y)G_N(x-z) T_2(y,z)\, dx\, dy\, dz\\
  \label{t3} &
I_3   = 2 \int_{[0,N]^4}
 G_N(x-y)G_N(z-t) T_2(x,z)T_2(y,t) \, dx\, dy\, dz\, dt \\
 \label{t4} & I_4= - 2\int_{[0,N]^2} G_N(x-y)^2 T_2(x,y)\, dx\, dy\\
 \label{t5} &
I_5=  4 \int_{[0,N]^3 } G_N(x-y)G_N(x-z)T_3(x,y,z)\, dx\, dy\, dz\\
\label{t6} &
I_6= - \int_{[0,N]^4 } G_N(x-y)G_N(z-t)T_4(x,y,z,t)\, dx\, dy\, dz\, dt,\end{align}
where $G_N(x)= -\log \left|2\sin \frac{\pi x}{N}\right|$ in dimension 1, resp. 
$G_N(x)=\frac{1}{2\pi} E_N(x) $ defined in \eqref{defE} in dimension 2.

\end{lem}

\begin{proof}
Expanding the square, we have
\begin{align}
\label{li1}& \(\sum_{i\neq j, a_i, a_j \in [0,N]} G_N(a_i- a_j)  \)^2 =
\sum_{i,j,k,l \ \mathrm{p.d.}} G_N(a_i- a_j)G_N(a_k- a_l) \\
\label{li2}& + \sum_{i,j,l\ \mathrm{p.d.}} G_N(a_i- a_j)G_N(a_i- a_l) +
\sum_{i,j,k \ \mathrm{p.d.}} G_N(a_i- a_j) G_N(a_k- a_i)\\
\label{li3}& +  \sum_{i,j,l \
\mathrm{p.d.}} G_N(a_i- a_j) G_N(a_j-a_l)+
 \sum_{i,j,k \
\mathrm{p.d.}} G_N(a_i- a_j) G_N(a_k-a_j)\\
\label{li4} & +
  \sum_{i\neq j} G_N(a_i-a_j)^2 + \sum_{i\neq j} G_N(a_i- a_j)  G_N(a_j-a_i),\end{align} where the sums are still taken over points in $[0,N]$, and p.d. stands for ``pairwise distinct''.
Since $G_N$ is even, it is clear  that all the sums in \eqref{li2} and \eqref{li3} are equal, and the sums in \eqref{li4} as well.
Using $k$-point correlation functions (cf. \eqref{corr}), we thus may write
\begin{multline}
\label{avecro}
\E \(\sum_{i\neq j, a_i, a_j \in [0,N]} G_N(a_i- a_j)\)^2\\
 = \int_{[0,N]^4}
 G_N(x-y)G_N(z-t) \ro_4(x,y,z,t) \, dx\, dy\, dz\, dt \\
 + 4\int_{[0,N]^3} G_N(x-y)\ro_3(x,y,z)\, dx\, dy\, dz+ 2\int_{[0,N]^2} G_N(x-y)^2 \ro_2(x-y)\, dx\, dy.
 \end{multline}

It is now convenient to express this in terms of the cluster functions $T_k$, using \eqref{reciprocal}, which yields
\begin{eqnarray*}
 \ro_2(x,y) & = & T_1(x)T_1(y)-T_2(x,y)=1-T_2(x,y)\\
 \ro_3(x,y,z)& =& 1-T_2(x,y)-T_2(x,z)- T_2(y,z)+T_3(x,y,z)\\
\ro_4(x,y,z,t)& =& 1-T_2(x,y) -T_2(x,z)-T_2(x,t) -T_2(y,z)-T_2(y,t)- T_2(z,t) \\
 & & + T_3(x,y,z)+T_3(x,y,t)+T_3(x,z,t)+T_3(y,z,t)\\
 & & + T_2(z,y)T_2(z,t)+T_2(x,z)T_2(y,t)+T_2(x,t)T_2(y,z) - T_4(x,y,z,t).\end{eqnarray*}
Substituting these relations into \eqref{avecro} and using that $\int_0^N G_N=0$, we obtain (writing the terms in the same order)
\begin{align} \label{avect}
& \E\(\sum_{i\neq j, a_i, a_j \in [0,N]} G_N(a_i- a_j)\)^2 =
\( \int_{[0,N]^2} G_N(x-y)T_2(x,y) \, dx\, dy \)^2 \\
\nonumber &  + 2 \int_{[0,N]^4]}
 G_N(x-y)G_N(z-t) T_2(x,z)T_2(y,t) \, dx\, dy\, dz\, dt \\
 \nonumber & - \int_{[0,N]^4 } G_N(x-y)G_N(z-t)T_4(x,y,z,t)\, dx\, dy\, dz\, dt\\
\nonumber & + 4 \int_{[0,N]^3 } G_N(x-y)G_N(x-z)T_3(x,y,z)\, dx\, dy\, dz\\
\nonumber & 
-4 \int_{[0,N]^4} G_N(x-y)G_N(x-z)T_2(y,z)\, dx\, dy\, dz\\
\nonumber & + 2 \int_{[0,N]^2} G_N(x-y)^2 \, dx\, dy- 2\int_{[0,N]^2} G_N(x-y)^2 T_2(x,y)\, dx\, dy.\end{align}
Similarly  (and as we have seen in the proof of Theorem \ref{th1}) we have
 $$\E\sum_{i\neq j}G_N(a_i-a_j)= - \int_{[0,N]^2} G_N(x-y)T_2(x,y)\, dx\, dy$$ and the result follows.\end{proof}

\subsection{The one-dimensional case}
\begin{theo} \label{th4}
For the sine-$\beta$ processes with $\beta=1,2,4$ as described above, we have 
$$\lim_{N\to \infty}\V (\Wc)=0.$$
\end{theo}

\begin{proof}
Since we already know from Proposition \ref{pro23} that $\lim_{N\to \infty} \E\Wc$ exists, and in view of \eqref{wc1d2}, it suffices
to show that 
\begin{equation}\label{vardef}
\lim_{N\to \infty}\frac{1}{N^2}\(  \E\( \sum_{i\neq j, a_i, a_j\in [0,N]} G_N(a_i- a_j)\)^2 - \(\E\sum_{i\neq j, a_i, a_j\in [0,N]} G_N(a_i- a_j)\)^2\) =0.\end{equation}

We apply Lemma \ref{prform} and
 now deal with all the terms $I_1$ to $I_6$ in \eqref{t1}--\eqref{t6}.
First,  we have
\begin{equation}
\label{t1r}I_1= 2\int_{[0,N]^2} G_N(x-y)^2 \, dx\, dy= 2N^2 \int_{[0,1]^2} \( \log \left|2\sin \pi (x-y)\right|\)^2\, dx\, dy.\end{equation}

For $I_2$, using the explicit expression for $G_N$, making  the change of variables $x'=x/N,  y'=y/N, t=y-z$,  and recalling that $T_2$ is translation-invariant since the process is,\footnote{hence by abuse of notation we write $T_2(x,y)=T_2(x-y)$ as for Theorem \ref{th1}}
we find
\begin{multline}\label{t2r}
I_2= - 4 \int_{[0,N]^4} G_N(x-y)G_N(x-z) T_2(y-z)\, dx\, dy\, dz\\
= - 4 N^2 \int_{[0,1]^2}\int_{[N(y-1), Ny]}  \log |2\sin \pi (x-y)|\log \left|2\sin \pi 
\left(x-y+\frac{t}{N}\right)\right| T_2(t)\, dx\, dy\, dt.\end{multline}
One may like to think that
$\int f(x-y+\frac{t}{N}) T_2(t)\, dt \to f(x-y)$ since $\int T_2=1$. However, $\log |2\sin\cdot  |$ is not  regular enough to apply this reasoning and $\int T_2$ may converge only conditionally. First notice that the cluster functions given in \eqref{t21}-\eqref{t22}-\eqref{t24} satisfy
\begin{equation}
\label{assv1} |T_2(v)|=O\(\frac{1}{|v|}\), \end{equation}
\begin{equation}
\label{assv2}
\int_{|v|>M} T_2(v)\, dv=O\(\frac{1}{M}\).\end{equation}

Pick two exponents $a,b>0 $ with $a+b<1$. Let us examine the  $t$ integral in the right-hand side of \eqref{t2r}. Assume first that $[-N^b, N^b]\subset [N(y-1), Ny]$ and that  $|x-y|>N^{-a}$.  Note that for this to be satisfied it suffices that 
\begin{equation}\label{restric}
(x,y)\in S_{N,a}:=\{ (x, y) \in \mr^2: 
N^{-a}<y<1-N^{-a}, \  |x-y|>N^{-a}\}.\end{equation}
By the mean value formula, we may write for some $|\theta|<N^b$
\begin{equation*}\log \left|2\sin \pi \left(x-y+\frac{t}{N}\right) \right|- \log \left|2\sin \pi (x-y)\right|= \frac{\pi \cos \pi (x-y+\frac{\theta}{N}) }{N\sin \pi (x-y+\frac{\theta}{N})}=O(N^{a-1})\end{equation*} since  we assumed $|x-y|>N^{-a}$.
Thus
\begin{multline}\label{contribpr}\int_{[-N^b, N^b]} \log |2\sin \pi (x-y)|\log \left|2\sin \pi \left(x-y+\frac{t}{N}\right)\right| T_2(t)\, dt\\
=\int_{[-N^b, N^b]} \log^2|2\sin \pi (x-y)| T_2(t)\, dt +O(N^{a-1} \log N) \\
= \log |2\sin \pi(x-y)| \bigl(1- O(N^{-b})\bigr) +O(N^{a-1} \log N)= 
\log |2\sin \pi(x-y)| +O(N^{-b} \log N).\end{multline}
where we have used \eqref{assv1}, then \eqref{assv2} and $|x-y|>N^{-a}$.
We then claim that if $|x-y|>N^{-a}$ we have 
\begin{equation}\label{reste1}
\int_{|t|>N^b} \log |2\sin \pi (x-y)|\log \left|2\sin \pi \left(x-y+\frac{t}{N}\right)\right| T_2(t)\, dt= o(1).\end{equation}
Assuming this, and combining with \eqref{contribpr}
we obtain that 
\begin{multline*}\int_{(x,y)\in [0,1]^2\cap S_{N, a} } \int_{[N(y-1), Ny]}  \log |2\sin \pi (x-y)|\log \left|2\sin \pi \left(x-y+\frac{t}{N}\right)\right| T_2(t)\, dx\, dy\, dt \\=
\int_{(x,y)\in [0,1]^2\cap S_{N,a}}\log^2 |2\sin \pi (x-y)|\, dx\, dy+o(1).\end{multline*}
But it is easy to check, since the integrals converge and \eqref{assv1} holds, that the contributions of the set where \eqref{restric} does not hold are $o(1)$ as $N \to \infty$. We may thus conclude that 
\begin{equation}\label{resultt1} 
I_2 = - 4 N^2\int_{(x,y)\in [0,1]^2}
\log^2 |2\sin \pi (x-y)|\, dx\, dy +o(N^2).\end{equation}
To finish with this $I_2$ term, it remains to prove \eqref{reste1}.
For $\beta=1,2$ this is immediately true since $T_2(v)=O(|v|^{-2})$.
For $\beta =4$, we notice that the same argument that was used above to restrict to $|x-y|>N^{-a}$ can be used to restrict to $|x-y + \frac{t}{N}|>N^{-c}$ (note that the initial integral is symmetric in $y$ and $z$). 
Inserting the formula for $T_2$ \eqref{t24}, and neglecting the $O(1/t^2)$ part of $T_2$,  we thus have to prove that 
$$\int_{|t|>N^{b}, |x-y+\frac{t}{N}| >N^{-c}} 
- \frac{\partial }{\p t}\frac{\sin 2\pi t}{2\pi t} \frac{1}{2\pi}\int_0^{2\pi t} \frac{\sin t}{t} \log \left|2\sin \pi \left(x-y+\frac{t}{N}\right)\right|=o(1).$$
We integrate by parts and find that the boundary terms are negligible,  and there remains to show that 
$$\int_{|t|>N^{b}, |x-y+\frac{t}{N}| >N^{-c}} 
\frac{\sin 2\pi t}{2\pi t} 
\frac{\partial }{\p t} \( \log \left|2\sin \pi \left(x-y+\frac{t}{N}\right)\right|\frac{1}{2\pi} \int_0^{2\pi t} \frac{\sin t}{t}\) \, dt=o(1).$$
If the derivative falls on the second factor, we are back to the $O(1/t^2)$ situation which gives a negligible term, and for the other term we use
$$\frac{\p}{\p t} \log \left|2\sin \pi \left(x-y+\frac{t}{N}\right)\right|= O(N^{c-1})$$ by explicit computation, which gives that the integral is $O(N^{c-1} \log N)=o(1)$.
This completes the treatment of $I_2$.

We turn to $I_3$. Using a similar change of variables, we may write this term
\begin{multline*}
I_3= 2N^2 \int_{[0,1]^2} \int_{[N(x-1),Nx]}\int_{[N(y-1), Ny]}
\log \left|2\sin\pi (x-y)\right|\log \left|2\sin \pi \left(x-y+\frac{v-u}{N}\right)\right| \\ 
\times T_2(u)
T_2(v)\, dx\, dy\, du\, dv.
\end{multline*}
Very similar manipulations to those above show that  $\log \left|2\sin \pi \left(x-y+\frac{v-u}{N}\right)\right| $ can be replaced by $\log |2\sin \pi (x-y)|$ with a $o(N^2)$ correction. This leads us to 
\begin{equation*}
I_3=2N^2 \int_{[0,1]^2 }\log^2 |2\sin \pi(x-y)|\, dx\, dy+o(N^2).
\end{equation*}
For $I_4$, we have
\begin{equation*}
I_4=\int_{[0,N]^2} \log^2 \left|2\sin \frac{\pi(x-y)}{N}\right|T_2(x-y)\, dx\, dy=o(N^2).\end{equation*}
Indeed, we may take away a $\delta N$-neighborhood of the diagonal, outside of  which $\log^2 \left|2\sin  \frac{\pi(x-y)}{N}\right|$ is bounded by $\log^2 N$ and $\int |T_2|$ is  controlled by $\log N $, using \eqref{assv1}. Thus the whole integral is controlled by $N\log^3 N=o(N^2)$.

Adding  the above  results we find that
$$I_1+I_2+I_3+I_4=o(N^2).$$
It remains to show that $I_5$ and  $I_6$ also
give $o(N^2)$ contributions.
The expressions $I_5$ and $I_6$ are estimated using explicit formulas for cluster functions of the sine-$\beta$ processes. First returning to  \eqref{k2}--\eqref{k1} we see that for $\beta=1,2,$ the entries of  $K^{(2)}$ and $K^{(1)}$ are 
$O\bigl(\frac{1}{1+|x-y|}\bigr)$. Combining with \eqref{tdet}--\eqref{tqdet}, it follows that $T_3(x,y,z)= O\(\frac{1}{1+|(x-y)(y-z)(z-x)|}\)$  in both these cases.
For \eqref{k4}, we have 
 \begin{equation}\label{entryk4}
 K^{(4)}(x,y)= \left(\begin{array}{lr}
 O\(\frac{1}{1+|x-y|} \)& O \( \frac{1}{1+|x-y|}\)\\
 O(1)  &  O \( \frac{1}{1+|x-y|}\)\end{array}\right).\end{equation}
  We thus obtain $$T_3(x,y,z)=O\left(\frac{1}{1+|(x-y)(y-z)|}\right)+O\left( \frac{1}{1+|(y-z)(z-x)|}\right)+O\left(\frac{1}{1+|(x-y)(z-x)|}\right).$$
In \eqref{t5} we may first (as above) remove a small neighborhood of the diagonals, off of which $|G_N(x-y)|$ and $|G_N(y-z)|$ are bounded by $O(\log N)$. It then remains to estimate $$\int_{[0,N]^3} |T_3(x,y,z)|\, dx\, dy\, dz.$$
Replacing $T_3$ by its above estimates,  and changing variables to $x+y+z$ and successively two out of $x-y$, $y-z$, and $z-x$, we find $\int_{[0,N]^3} |T_3(x,y,z)|\, dx\, dy\, dz\le O(N \log^3 N)$, and $I_5=o(N^2)$.

We finally turn to $I_6$.
The formula for $T_4 $ is given by \eqref{tdet}--\eqref{tqdet}. Comparing to \eqref{k2}--\eqref{k1}, we see that for $\beta=1,2$, we have $$T_4(x,y,z,t)= O\left( \frac{1}{1+|(x-y)(y-z)(z-t)(t-z)|}\right).$$ The same reasoning as for $I_5$ gives $I_6=o(N^2)$.

For $\beta=4$, in the formula for $T_4$ obtained with \eqref{k4}, in view of \eqref{entryk4}, there are terms which a priori have insufficient decay: they are terms of the form
$$ \frac{\p}{\p x_1} \frac{\sin\pi(x_1-x_2)}{2\pi(x_1-x_2)} \frac{\p }{\p x_3}
\frac{\sin 2\pi(x_3-x_4)}{2\pi(x_3-x_4)} \int_0^{2\pi(x_2-x_3)}\frac{\sin t}{t}\, dt \int_0^{2\pi(x_4-x_1)}\frac{\sin t}{t}\, dt,
$$ where $(x_1,\dots, x_4)$ is a permutation of the variables $x,y,z,t$.
This leads to two different types of integrals
\begin{multline}\label{typea}
\int_{[0,N]^4}
\log \left|2\sin \frac{\pi(x-y)}{N}\right|
\log \left|2\sin \frac{\pi(z-t)}{N}\right|  \\ \times
\int_0^{2\pi(x-y)}\frac{\sin s}{s}\, ds\int_0^{2\pi(z-t)} \frac{\sin s}{s}\,ds \cdot \frac{\p }{\p x} \frac{\sin 2\pi(x-z)}{2\pi(x-z)}\frac{\p }{\p y} 
\frac{\sin 2\pi(y-t)}{2\pi(y-t)}\, dx\,dy\,dz\,dt\end{multline}
and
\begin{multline}\label{typeb}
\int_{[0,N]^4}
\log \left|2\sin \frac{\pi(x-y)}{N}\right|
\log \left|2\sin \frac{\pi(z-t)}{N}\right| \\ \times
\int_0^{2\pi(x-z)}\frac{\sin s}{s}\, ds\int_0^{2\pi(y-t)} \frac{\sin s}{s}\,ds \cdot \frac{\p }{\p x} \frac{\sin 2\pi(x-y)}{2\pi(x-y)}\frac{\p }{\p z}
\frac{\sin 2\pi(z-t)}{2\pi(z-t)}\, dx\,dy\,dz\,dt.\end{multline}

For \eqref{typea}, we may again restrict the domain to $|x-y|>N^a$ with $0<a<1$. 
Then, integrating by parts in $x$ gives boundary terms which are negligible, 
and a new integrand with extra decay, involving $\frac{\p}{\p x} \int_0^{2\pi(x-y)} \frac{\sin s}{s}\, ds$ or $\frac{\p }{\p x} \log \left|2\sin \frac{\pi(x-y)}{N}\right|$. This leads again to $o(N^2)$ contributions.

For \eqref{typeb}, we may first restrict the integral to $|x-z|>N^a$ and $|y-t|>N^a$, using arguments as above. Then, we may replace $
\int_0^{2\pi(x-z)}\frac{\sin s}{s}\, ds$ and $\int_0^{2\pi(y-t)} \frac{\sin s}{s}\,ds $ by $\frac{\pi}{2}sgn(x-z) $ and $\frac{\pi}{2}sgn(y-t)$ respectively, making only a $o(N^2)$ error.
Then note that the integrand in $x-y$ is a locally odd function, so we can remove domains $|x-y|<N^b$, $|z-t|<N^b$ from the integration domain. Finally, integration by parts in $x$ gives additional decay, yielding $o(N^2)$ contribution. 
We conclude that  $I_6=o(N^2)$, and the result follows.
\end{proof}

\begin{coro}For all point processes with finite $\lim_{N\to\infty} \E(\Wc)$ and $\lim_{N\to \infty} \V(\Wc)=0$, 
 $$\Wc \to \lim_{N\to \infty} \E\Wc \quad \text{as } \ N \to \infty$$
 in $L^2(\Omega)$ and thus in probability.\end{coro}
 The proof is immediate.
 Processes satisfying these assumptions and having
 different values for $\lim_{N\to \infty}\E\Wc$ are thus mutually singular, such as $\beta$-sine processes with different $\beta\in\{1,2,4\}$.

\subsection{The two-dimensional case}

\begin{theo}
For the determinantal random point process with kernel \eqref{kginibre} we have 
$$\lim_{N\to \infty}\V
(\Wc)=0.$$\end{theo}

\begin{proof} The starting point is again Lemma \ref{prform}.
We note that in view of  \eqref{kginibre}, \eqref{rodet} and \eqref{deftn}, all the cluster functions for that process are exponentially 
decreasing when viewed as functions of pairwise distances between arguments. 
We start with the term \eqref{t2}. Replacing $G_N$ by $\frac{1}{2\pi} E_N$, using the translation invariance,  and changing variables as before ($y-z=u$), we find
\begin{multline*}
I_2+2I_1=\\ \frac{2N^2}{\pi}\int_{[0,1]^2}\int_{Ny-[0,N]^2}E_N(N(x-y))\( E_N(N(x-y))- E_N(N(x-y+u/N))\) T_2(u)dx\, dy\,  du.
\end{multline*}
Noting that in view of the definition of $E$ (cf. \eqref{eis}--\eqref{defE}) we have $E_N(Nx)=E_1(x)$ and from \eqref{asE}, $E_1$ behaves like $C\log|x|$ near $x=0$ (and similarly near 
points of the lattice $\mz^2$).
Given $\eta>0$ there thus exists $\delta>0$ such that 
 $$\int_{[0,1]^2 \cap \{|x-y- \mz^2|<\delta\}} \int_{Ny-[0,N]^2}E_1(x-y)\( E_1(x-y))- E_1(x-y+u/N)\) T_2(u)dx\, dy\,  du<\eta.$$
 On the other hand, still in view of the definition \eqref{eis}--\eqref{defE}, $E_1$ is uniformly continuous away from $\mz^2$, hence 
 we may write, as $N \to \infty$,  $$
 \int_{[0,1]^2 \backslash \{|x-y- \mz^2|<\delta\}} \int_{Ny-[0,N]^2}E_1(x-y)\( E_1(x-y))- E_1(x-y+u/N)\) T_2(u)dx\, dy\,  du=o(1).$$
 Since this is true for any $\eta$, it follows that 
$$I_2+ 2I_1=o(N^2).$$
Similarly,
\begin{multline*}
I_3 -I_1=
\\
\int_{[0,1]^2 }\int_{Nx-[0,N]^2}\int_{Ny-[0,N]^2} E_1(x-y)\(E_1\left(x-y+\frac{u-v}{N}\right)   - E_1(x-y)\) T_1(u)T_2(v)\, dx\, dy\, du\, dv.\end{multline*}
The same reasoning shows that this is $o(N^2)$.
For $I_4$,  the change of variables $x'=x/N$ and $y'=y/N$ yields
$$I_4 =-2N^2\int_{[0,1]^2} E_1(x-y)^2 T_2(N(x-y))\, dx\, dy.$$
We may take out a $\delta$ neighborhood of the diagonal and its translates by $\mz^2$, off of which $E_1(x-y)$ can be bounded by $C \log |x-y|$ and $T_2(\frac{x-y}{N})$ by $e^{-C N^2 \delta^2}$. The whole term is thus $o(N^2)$.

We turn to \eqref{t5}.
From \eqref{tdet} and \eqref{kginibre} we find that $|T_3(x,y,z)|\le e^{-C(|x-y|^2+|y-z|^2 + |x-z|^2)}.$
As above we have
$$I_5= 4\int_{[0,1]^3}E_1(x-y)E_1(x-z)T_2(Nx,Ny,Nz)\, dx\, dy\, dz
$$
and as above we may take out  $\delta$-neighborhoods $|x-y|<\delta$ or $|x-z|<\delta$ or $|y-z|<\delta$ (and their translates by $\mz^2$), outside of which the $E_1$ terms are bounded by $\log$'s 
and $T_3$ by $e^{-C N^2 \delta^2}$. The whole term is thus $o(N^2)$.

A very similar reasoning applies to $I_6$.

Combining all these, we find the result.
\end{proof}

\begin{remark} In view of the proof above, the same result holds for any process such 
that the cluster functions decay sufficiently fast away from diagonals, for example exponentially.
\end{remark}

\section{Miscellaneous computations}\label{sc:misc}
In this section we gather various additional computations of expectations and additional facts.
\subsection{Operations on processes}
In this subsection, we examine the effect on $\lim_{N\to \infty} \E\Wc$ of two common operations on independent processes: superposition and decimation (see \cite{dvj}).
\begin{pro}\label{pro26}
Let $\X_1,\dots , \X_M$ be $M$  independent translation invariant 
point processes with density $1$  and two-point correlation functions $\bigl\{\ro_2^{(i)}\bigr\}_{1\le i\le M}$, satisfying the assumptions of Theorem \ref{th1}. 

Assume that  $\lim_{N\to \infty}\E \Wc(\X_i)<\infty$ for $i=1,\dots,M$. Let $\X$ denote
the superposition of independent processes $\bar{\X}_i$, where $\bar{\X}_i$ denotes the image of the process $\X_i$ under the dilation by factor $M$ of the line. Then, with obvious notation,
$$\lim_{N\to \infty} \E\Wc  (\X)= \log M + \frac{1}{M} \sum_{i=1}^M  \lim_{N\to \infty} \E \Wc (\X_i).$$
\end{pro}
\begin{proof}
 Let $T_2^{(i)}$ be the  second cluster functions corresponding to $\X_i$.
 Let $\bar{\ro}_2^{(i)}$ now denote the second correlation function for the process $\bar{\X_i}$,  which  has density $1/M$.  We have
 $\bar{\ro}_2^{(i)}(x,y)= \frac{1}{M^2}\ro_2^{(i)}(\frac{x}{M}, \frac{y}{M})$.
 The process $\X$ clearly  has density $1$, and its second correlation function is
 $$\ro_2(x,y)= \sum_{i \neq j \in [1,M]} \ro_1^{(i)}  \ro_1^{(j)}
+ \sum_{i=1}^M \bar{\ro}_2 (x,y)   =M(M-1) \frac{1}{M^2}+ \frac{1}{M^2} \sum_{i=1}^M\ro_2^{(i)}\left(\frac{x}{M}, \frac{y}{M}\right).$$
We also denote by $T_2$  the corresponding second cluster function. We thus
 have  $$T_2(v)= \frac{1}{M} - \frac{1}{M^2} \sum_{i=1}^M \(1-  T_2^{(i)} \left(\frac{v}{M}\right)\)= 
\frac{1}{M^2} \sum_{i=1}^M T_2^{(i)} \left(\frac{v}{M}\right),$$ and it easily follows, with a change of variables,  that
$\int_{- \infty}^\infty T_2(v)\, dv= 1.$
In addition,
\begin{multline}
\int_{-\infty}^{\infty} T_2(v) \log |2\pi v|\, dv=
\int_{-\infty}^\infty    \frac{1}{M^2} \sum_{i=1}^M T_2^{(i)} \left(\frac{v}{M}\right)\log |2\pi v|\, dv\\
= \frac{1}{M}\sum_{i=1}^M \int_{-\infty}^\infty T_2^{(i)} (s) \log |2\pi  Ms|\, ds=
\log M + \frac{1}{M} \sum_{k=1}^M \int_{-\infty}^\infty T_2^{(i)}(v) \log|2\pi v|\, dv.
\end{multline}
The result follows easily using Theorem \ref{th1}.
\end{proof}

In dimension 2, reproducing the proof, but replacing the dilations by factor $M$ by dilations by factor $\sqrt{M}$, we obtain instead the result:
$$\lim_{N\to \infty} \E\Wc  (\X)= \hal\log M + \frac{1}{M} \sum_{i=1}^M  \lim_{N\to \infty} \E \Wc (X_i).$$

For example, superposing two independent processes with the same second correlation function 
leads to an increase of $\lim_{N\to \infty}\E \Wc$ by $\log 2$  in dimension 1 and 
$\hal \log 2 $ in dimension 2. Superposition can thus be seen as ``increasing the disorder".

We next turn to decimation.

One can define a `random decimation' of a process by erasing points at random with probability 1/2. The second correlation function  then transforms into
$$R_2(x,y)= \frac{1}{4}\ro_2 (x,y).$$ But the space needs to be rescaled by a factor 2 in order 
to maintain a  density one, so the correlation function after that is
$$\ro_2'(x,y)= \ro_2(2x, 2y).$$
It is clear in view of Theorems \ref{th1} and \ref{th2} that if a process has finite $\lim_{N\to \infty}\E\Wc$, its decimation will not, since the condition $\int T_2=1$ will be destroyed by this operation.

On the other hand, one can define a `deterministic decimation' by erasing every even (or odd)
point of an (ordered) random point configuration followed by rescaling of the space to keep
the density at $1$. While we cannot say anything about this operation in general, one can observe 
what it does in a couple of cases. 

It is known, see e.g. \cite[page 66]{agz}, that the $\beta=2$ sine process is the 
deterministic decimation of superposition of two $\beta=1$ sine processes, or symbolically
$$
(\text{sine } \beta=2) = \text{decimation}((\text{sine } \beta=1)\sqcup (\text{sine } \beta=1)).
$$
From Proposition~\ref{pro23} we know that $\lim_{N\to \infty} \E\Wc$ is $1-\gamma$ for the left-hand side,
and Proposition~\ref{pro26} says that $\lim_{N\to \infty} \E\Wc$ is $2-\gamma$ for 
$(\text{sine } \beta=1)\sqcup (\text{sine } \beta=1)$. Thus, the deterministic decimation decreased
the value of  $\lim_{N\to \infty} \E\Wc$ by 1. 

Similarly, 
$$
(\text{sine } \beta=4) = \text{decimation}((\text{sine } \beta=1)),
$$
and we see that the decimation decreased the value of $\lim_{N\to \infty} \E\Wc$ from $2-\gamma-\log 2$ to
$\frac32-\gamma-\log 2$.

 \subsection{Discrete $\beta=2$ sine process}\label{sec5}

The $\beta=2$ discrete sine process was first obtained in \cite{boo} as the bulk scaling limit
of the Plancherel measure for symmetric groups, and it was shown in \cite{bkmm} to be the
a universal local scaling limit for a broad family of discrete probabilistic models 
of random matrix type with $\beta=2$. The goal of this section is to compute 
$\lim_{N\to\infty}\E\Wc$ for the suitably scaled discrete sine process embedded into the real
line. By the construction, this provides an interpolation between the case of the perfect 
lattice, for which $\Wc\equiv 0$, and the case of the continuous $\beta=2$ sine process
treated in the previous sections. 

 Let $\ro \in (0,1)$. The discrete sine process with density $\ro$ is a random point process on $\mz$ with the correlation functions ($k\ge 1$)
 $$\ro_k(x_1, \dots, x_k)= \det \left[\frac{\sin \pi \ro(x_i-x_j)}{\pi(x_i-x_j)}\right]_{i,j=1}^k.$$
\begin{pro}
 Embed $\mz $ into $\mr$ via $n \mapsto \ro n$; this turns the discrete sine process of density $\rho$
into a random point process on $\mr$ with density $1$.
 For the latter process we have
\begin{equation}\label{formdiscsin}\lim_{N\to \infty}\E\Wc= \ro \log \ro + \frac{2}{\ro} \sum_{u=1}^\infty \( \frac{\sin (\ro \pi u)}{\pi u}\)^2 \log (2\pi \ro u).
\end{equation}
\end{pro}

\begin{proof}
For the calculation, we assume that $\frac{N}{\ro}$ is an integer (the same argument however 
should hold without this assumption by examining more carefully error terms).
The calculation is then can be viewed as a discrete version of that of Theorem \ref{th1}.
First, by definition \eqref{wc1d2} of $\Wc$, we have
\begin{equation*}
\E \Wc= - \frac{1}{N} \sum_{i\neq j \in [1, N/\ro]} \ro_2 (i,j)
 \log \left|2\sin \frac{\pi \ro (i-j)}{N} \right|+\log N\end{equation*}
 and since $\ro_2(i,j)=\ro^2  - \left( \frac{\sin \pi \ro (i-j)}{\pi(i-j)}\right)^2$ 
we find
 \begin{multline}\label{Wcsin}
 \E \Wc= \frac{1}{N}  \sum_{i \neq j \in [1, N/\ro] } \( \frac{\sin \pi \ro (i-j)}{\pi(i-j)}   \)^2    \log \left|2\sin \frac{\pi \ro (i-j)}{N}\right| + \log N \\ - \frac{\ro^2}{N} \sum_{i \neq j \in [1, N/\ro] } \log \left|2\sin \frac{\pi \ro (i-j)}{N}\right| .\end{multline}
  Let us first examine the contribution of the last sum.
  From the knowledge of $\Wc(\mz)=0$, we know that for $K\in \mz$
 we have
 $$\lim_{K\to \infty}  \frac{1}{K}\sum_{i\neq j \in [1,K]} \log \left|2\sin \frac{ \pi (i-j)}{K}\right|+\log K=0.$$
 Applying this to $K = N/\ro$ we find that
\begin{equation}\label{knr}
\frac{\ro^2}{N} \sum_{i \neq j \in [1, N/\ro] } \log \left|2\sin \frac{\pi (i-j)}{N/\ro}\right| = - \ro \log N/\ro + o(1).\end{equation}
We next turn to the first two terms in \eqref{Wcsin}.  As in Theorem \ref{th1}, we expect only the near diagonal terms to contribute, so that
$$\log \left|2\sin \frac{\pi \ro(i-j)}{N}\right|\sim \log \left|\frac{2\pi \ro (i-j)}{N}\right|=\log |2\pi \ro (i-j)|- \log N.$$
Inserting this and \eqref{knr} into \eqref{Wcsin} we find
\begin{equation}\label{Wcsin2}
\E \Wc\sim\frac{1}{N}  \sum_{i \neq j \in [1, N/\ro] } 
\left( \frac{\sin \pi \ro (i-j)}{\pi(i-j)}   \right)^2   
( \log \left|2\pi \ro (i-j)\right| -\log N ) -
\ro \log N/\ro + \log N +o(1).
\end{equation}
 We first focus on
 \begin{align}\label{lonn}
 \frac{- \log N }{N}  \sum_{i \neq j \in [1, N/\ro] } \( \frac{\sin \pi \ro (i-j)}{\pi(i-j)}   \)^2 &= & - \frac{\log N}{N} \sum_{u=1}^{N/\ro -1} (N/\ro - u) \frac{2\sin^2 (\pi \ro u)}{\pi^2 u^2}\\ \nonumber &=& - \frac{\log N}{\ro} \sum_{u=1}^{N/\ro -1}
 \frac{1-\cos 2\pi \ro u}{\pi^2 u^2}+o(1).\end{align}
Indeed, we can bound $ \sum_{i=1}^{N/\ro } u \frac{2 \sin^2 (\pi \ro u)}{\pi^2 u^2} $ by $\sum_{i=1}^{N/\ro} \frac{1}{u}=O(\log N/\ro)$ and this multiplied by $\frac{\log N}{N}$ is negligible as $N\to + \infty$.
The last sum then appears as a Fourier series and can be computed explicitly, which leads to
\begin{equation}\label{rm1} - \frac{\log N}{\ro} \sum_{u=1}^{N/\ro -1}
 \frac{1-\cos 2\pi \ro u}{\pi^2 u^2}= \frac{\log N}{\ro} \(\frac{\pi^2}{6}- \frac{1}{12}(12 \pi^2 \ro^2 - 12 \pi \ro + 2\pi^2)\) = (\ro -1)\log N.\end{equation}
We next turn to
\begin{equation}\label{rm2}\frac{1}{N}  \sum_{i \neq j \in [1, N/\ro] }
 \( \frac{\sin \pi \ro (i-j)}{\pi(i-j)}   \)^2    \log |2\pi \ro (i-j)|= \frac{2}{N} \sum_{u=1}^{N/\ro-1} (N/\ro -u) \(\frac{ \(
\sin \pi \ro u\)^2}{\pi^2 u^2}\)^2\log (2\pi \ro u).\end{equation}
 Again, the term containing $u$ can be neglected since it is bounded by $\frac{1}{N} \sum_{u=1}^{N/\ro} \frac{1}{u} \log (2\pi \ro u)\le O(\frac{\log^2 N}{N})=o(1)$.
 Combining \eqref{Wcsin2}--\eqref{rm2} and letting $N \to \infty$, we finally arrive at \eqref{formdiscsin}.
\end{proof}

The graph of $\lim_{N\to \infty}\E\Wc$ in \eqref{formdiscsin} is presented in Fig. \ref{fig2}, it shows a function decreasing 
from $1-\gamma$ at $\ro =0$ to $0$ at $\ro =1$, as expected.
\begin{figure}\label{fig2}
\begin{center}
\includegraphics[height=7cm]{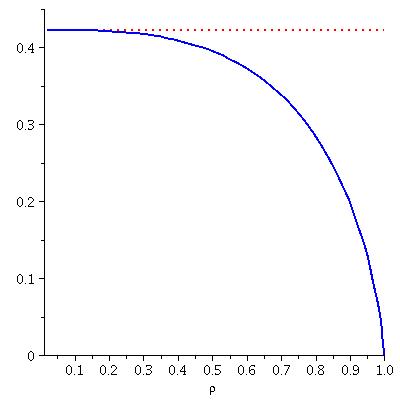}\caption{Numerical evaluation of $\lim_{N\to \infty}\E \Wc$ for the discrete sine process}\end{center}
\end{figure}

\noindent
{\sc Alexei Borodin}\\
Massachusetts Institute of Technology\\
Department of Mathematics\\
77 Massachusetts Avenue, 
Cambridge, MA 02139-4307, USA.\\
{\tt borodin@math.mit.edu}
\\ \\
{\sc Sylvia Serfaty}\\
UPMC Univ  Paris 06, UMR 7598 Laboratoire Jacques-Louis Lions,\\
 Paris, F-75005 France ;\\
 CNRS, UMR 7598 LJLL, Paris, F-75005 France \\
 \&  Courant Institute, New York University\\
251 Mercer st, NY NY 10012, USA\\
{\tt serfaty@ann.jussieu.fr}

\end{document}